\documentclass[10pt,A4paper,fleqn]{amsart}
\usepackage{mathrsfs}
\usepackage{amsfonts}
\usepackage{mathbbold}
\usepackage{txfonts}
\usepackage{amsmath}
\usepackage{amssymb}
\usepackage{amsthm}
\usepackage[pdftex]{graphicx}
\usepackage[toc,page,title,titletoc,header]{appendix}
\usepackage{geometry}
\usepackage{upgreek}
\usepackage[T1]{fontenc}
\usepackage{color}
\usepackage{hyperref}
\usepackage[all]{xy}
\usepackage[french,english]{babel}

\theoremstyle{plain}
\newtheorem{theo}{Theorem}[section]
\newtheorem*{theo*}{Theorem}
\newtheorem{prop}{Proposition}[section]

\newtheorem{cor}{Corollary}[section]
\newtheorem*{lem*}{Lemma}

\theoremstyle{definition}
\newtheorem{lem}{Lemma}[section]

\newtheorem{example-condition}{Example-Condition}[section]

\theoremstyle{remark}

\newtheorem*{ack}{Acknowledgements}
\newtheorem*{rem*}{Remark}

\newcommand{\R}{\mathbb{R}}
\newcommand{\T}{\mathbb{T}}
\newcommand{\Z}{\mathbb{Z}}
\newcommand{\C}{\mathbb{C}}

\newcommand{\N}{\mathbb{N}}

\begin{document}

\bibliographystyle{unsrt}

\abstract
By application of KAM theorem to Lidov-Ziglin's global study of the quadrupolar approximation of the spatial lunar three-body problem, we establish the existence of several families of quasi-periodic orbits in the spatial lunar three-body problem.
\endabstract

\title{Quasi-periodic Solutions of the Spatial Lunar Three-body Problem}

\date\today
\author{Lei Zhao}
\address{ASD, IMCCE, Observatoire de Paris, 77 Avenue Denfert-Rochereau 75014 Paris, France / Université Paris Diderot }
\email{zhaolei@imcce.fr}
\maketitle

\tableofcontents

\section{Introduction}
We consider the spatial lunar three-body problem, in which two bodies form a binary system (the inner pair) which is slightly perturbed by the third far-away body. By considering two relative positions, the study of the three-body problem is reduced to the study of the motions of these relative positions. The relative position of the inner pair and the relative position of the third body with respect to the mass center of the inner pair thus   describe almost Keplerian orbits that we suppose to be elliptic, with instantaneous semi major axes $a_{1} << a_{2}$. We decompose the Hamiltonian $F$ as 
$$F=F_{Kep}+F_{pert},$$
in which $F_{Kep}$ describes two uncoupled elliptic Keplerian motions, and $F_{pert}$ (which causes the secular evolutions of the elliptic orbits) is small provide that $a_{1}$ is sufficiently small compared to $a_{2}$. 

Due to the proper-degeneracy of the Kepler problem (all of its bounded orbits are closed), to study the dynamics of $F$ by a perturbative study of $F_{Kep}$, it is important to study the higher order effect of the (slow) secular evolution of the Keplerian orbits given by the perturbation. In the lunar three-body problem, the smallness of $a_{1}$ with respect to $a_{2}$ implies the lack of all lower-order resonances of the two Keplerian frequencies. As a result, the corresponding elimination procedure does not require any arithmetic conditions on the two Keplerian frequencies, and the secular evolution of the Keplerian orbits is approximately given by the \emph{secular systems}, which are the successive averaged systems of $F_{pert}$ over the two fast Keplerian angles. 

In the planar case, the 4-degree-of-freedom secular systems are invariant under the Keplerian $\T^{2}$-action and the $\hbox{SO(2)}$-action of rotations, and are thus integrable, which allows J. Féjoz to carry out a global study of their dynamics and establish families of quasi-periodic orbits of the planar three-body problem in \cite{QuasiMotionPlanar}. The integrability of the secular systems is no longer guaranteed in the spatial case, which has 6 degrees of freedom with (only) a $\T^{2} \times \hbox{SO(3)}$-symmetry. Nevertheless, if we expand the secular systems in powers of the semi major axis ratio $\alpha=\dfrac{a_{1}}{a_{2}}$ (the natural small parameter of the lunar problem), then the first non-trivial coefficient $F_{quad}$, \emph{i.e.} the \emph{quadrupolar system}, is noticed by Harrington \cite{Harrington1968} to have an additional $\hbox{SO(2)}$-symmetry giving rise to an additional first integral $G_{2}$ (the norm of the outer angular momentum), and is thus integrable. Integrable approximating systems can thus be obtained by a further single-frequency averaging procedure over the conjugate angle of $G_{2}$.

Aside from the degenerate inner ellipses, the dynamics of the quadrupolar system is studied globally by Lidov-Ziglin \cite{LidovZiglin} (see also \cite{FerrerOsacar}, \cite{FaragoLaskar}). By verifying of the required non-degeneracy conditions and applying a sophisticated iso-chronous version of the KAM theorem, we confirm that almost every invariant Lagrangian tori and some normally elliptic invariant isotropic tori of the quadrupolar system give rise to irrational invariant tori of the spatial lunar three-body problem after being reduced by the $\hbox{SO(3)}$-symmetry (Theorems \ref{FarFromCollisionMotions} and \ref{elliptictori}). These results extend some former results of Jefferys-Moser \cite{JefferysMoser} of the spatial lunar three-body problem concerning normally hyperbolic invariant isotropic tori to the normally elliptic tori and the Lagrangian tori. A theorem of J. P{\"o}schel confirms that each of these Lagrangian tori is accumulated by periodic orbits in this reduced system (Theorem \ref{AccumulatingPeriodicOrbits}).

By application of KAM theorems, the existence of various families of quasi-periodic solutions of the Newtonian $N$-body problem, planar or spatial, was shown in \emph{e.g.} \cite{Arnold1963}, \cite{JefferysMoser}, \cite{Lieberman}, \cite{Robutel1995}, \cite{BiascoChierchiaValdinoci}, \cite{QuasiMotionPlanar}, \cite{FejozStability}, \cite{ChierchiaPinzariPlanetary}, \cite{MeyerPalacianYanguas2011}. Due to the frequent non-integrability of the approximating systems, most of these work are local studies in some neighborhoods of the phase space, with the only exception of \cite{QuasiMotionPlanar}. Based on a global study of the secular dynamics of the lunar case, our results also extend Féjoz's families of quasi-periodic solutions of the planar secular systems to the spatial lunar case.

\section{Hamiltonian Formalism of the Three-body Problem} \label{Section: Formulation}
\subsection{The Hamiltonian System}
We study the Hamiltonian system with phase space
$$\Pi:=\left\{(p_{j}, q_{j})_{j=0,1,2}=(p_{j}^{1}, p_{j}^{2}, p_{j}^{3}, q_{j}^{1}, q_{j}^{2}, q_{j}^{3}) \in (\R^{3} \times \R^3)^3 |\,  \forall 0 \leq j \neq k \leq 2, q_j \neq q_k \right\}, $$
(standard) symplectic form 
$$\sum^{2}_{j=0} \sum^{3}_{l=1} d p_j^l \wedge d q_j^l,$$
and the Hamiltonian function
$$F=\dfrac{1}{2} \sum_{0 \le j \le 2} \dfrac{\|p_j\|^2}{m_j} -  \sum_{0 \le j < k \le 2} \dfrac{m_j m_k}{\|q_j- q_k\|},$$
in which $q_0,q_1,q_2$ denote the positions of the three particles, and $p_0,p_1,p_2$ denote their conjugate momenta respectively. The Euclidean norm of a vector in $\R^{3}$ is denoted by $\|\cdot\|$. The gravitational constant has been set to $1$.

\subsection{Jacobi Decomposition}

The Hamiltonian $F$ is invariant under translations in positions. To symplectically reduce the system by this symmetry, we {switch} to the \emph{Jacobi coordinates} $(P_i, Q_i),i=0, 1, 2,$ with
\begin{equation*} 
\left\{
\begin{array}{l} P_0=p_0+p_1+ p_2 \\ P_1=p_1+ \sigma_1 p_2\\ P_2 = p_2
\end{array} \right.
\hbox{ \phantom{aaaaaaaqqqqaa}}
\left\{
\begin{array}{l} Q_0=q_0 \\ Q_1=q_1- q_0 \\ Q_2=q_2-\sigma_0 q_0-\sigma_1 q_1,
\end{array} \right.
\end{equation*}
in which
$$\dfrac{1}{\sigma_0}=1+\dfrac{m_1}{m_0},  \dfrac{1}{\sigma_1}=1+\dfrac{m_0}{m_1}.$$
The Hamiltonian $F$ is thus independent of $Q_{0}$ due to the symmetry. We fix $P_{0}=0$ and reduce the translation symmetry by eliminating $Q_{0}$. 
In the (reduced) coordinates $(P_i, Q_i),i=1, 2$, the function $F=F(P_{1}, Q_{1}, P_{2}, Q_{2})$ describes the motions of two fictitious particles. 

We further decompose the Hamiltonian $F(P_{1}, Q_{1}, P_{2}, Q_{2})$ into two parts $F=F_{Kep}+F_{pert}$, where the \emph{Keplerian part} $F_{Kep}$ and the \emph{perturbing part} $F_{pert}$ are respectively
\begin{align*}
&F_{Kep}=\dfrac{\|P_1\|^2}{2 \mu_1}+\dfrac{\|P_2\|^2}{2 \mu_2}-\dfrac{\mu_1 M_1}{\|Q_1\|}-\dfrac{\mu_2 M_2}{\|Q_2\|},
  \\&F_{pert}=-\mu_1 m_2\Bigl[\dfrac{1}{\sigma_o}\bigl(\dfrac{1}{\|Q_2-\sigma_0 Q_1\|}-\dfrac{1}{\|Q_2\|}\bigr)+\dfrac{1}{\sigma_1}\bigl(\dfrac{1}{\|Q_2+\sigma_1 Q_1\|}-\dfrac{1}{\|Q_2\|}\bigr)\,\Bigr],
 \end{align*} \label{Not: pert part}
 with (as in \cite{QuasiMotionPlanar})
 \begin{align*}&\dfrac{1}{\mu_1}=\dfrac{1}{m_0}+\dfrac{1}{m_1}, \, \dfrac{1}{\mu_2}=\dfrac{1}{m_0+m_1}+\dfrac{1}{m_2},\\ & M_1=m_0+m_1, M_2=m_0+m_1+m_2.
 \end{align*} 

We shall only be interested in the region of the phase space where $F=F_{Kep}+F_{pert}$ is a small perturbation of a pair of Keplerian elliptic motions.  Let $a_1,a_2$ be the semi major axes of the (instanteneous) inner and outer ellipses respectively. We shall further restrict us to the lunar three-body problem, characterized by the fact that the \emph{ratio of the semi major axes} by $\alpha=\dfrac{a_1}{a_2}$ is sufficiently small. 

\subsection{Delaunay Coordinates}\label{Subsection: Delaunay Coordinates}
The Delaunay coordinates 
$$(L_i,l_i,G_i,g_i,H_i,h_i),i=1,2$$ for both elliptic motions are defined as the following:
\begin{equation*} 
\left\{
\begin{array}{ll}L_i=\mu_i \sqrt{M_i} \sqrt{a_i}   & \hbox{circular angular momentum}\\ l_i  &\hbox{mean anomaly}\\ G_i = L_i \sqrt{1-e_i^2} &\hbox{angular momentum} \\g_i &\hbox{argument of pericentre} \\ H_i=G_i \cos i_i &\hbox{vertical component of the angular momentum} \\ h_i &\hbox{ longitude of the ascending node},
\end{array}\right.
\end{equation*}

\begin{figure}
\centering
\includegraphics[width=3in]{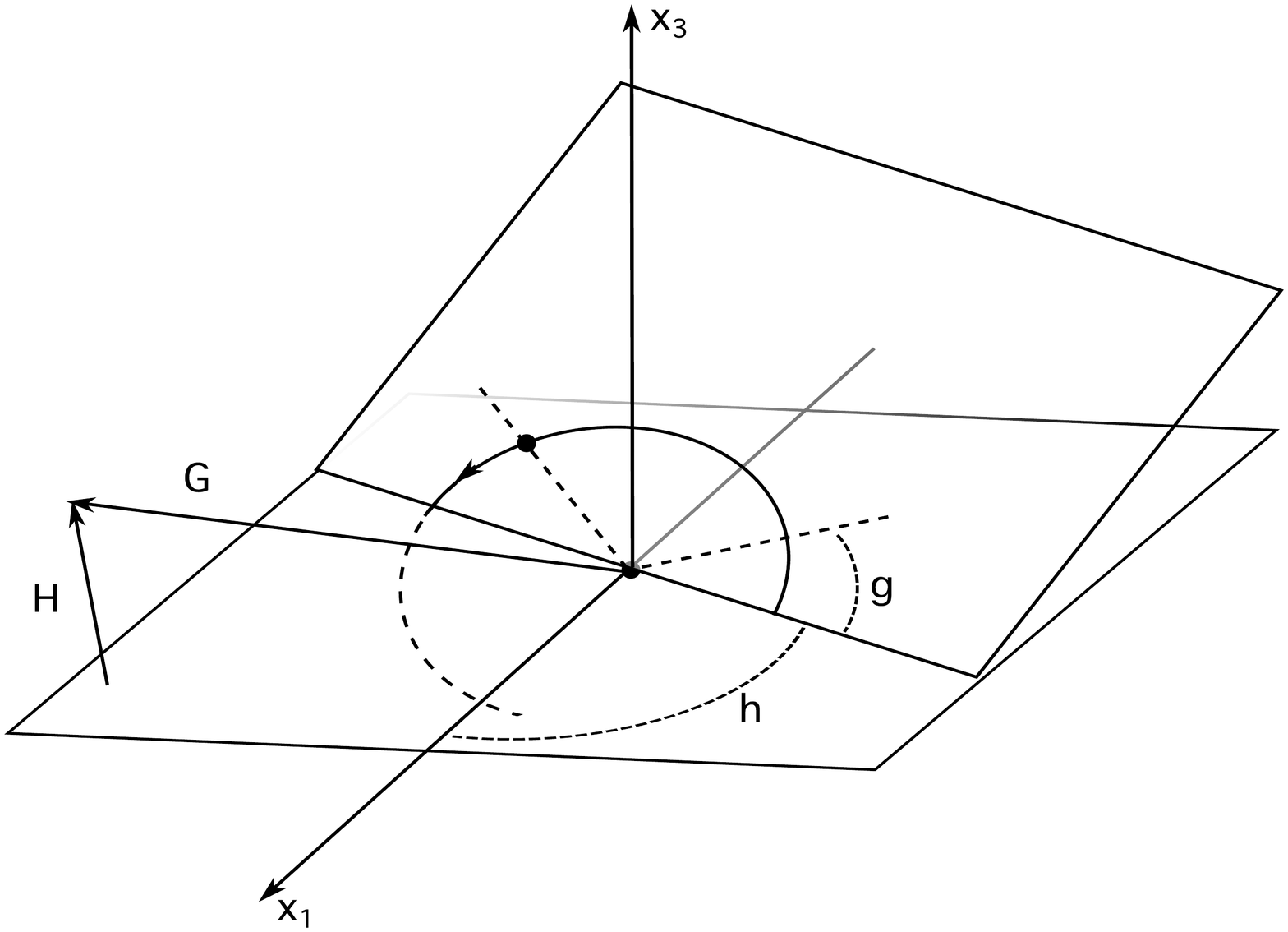}
\caption{Some Delaunay Variables}
\end{figure}
in which $e_{1}$, $e_{2}$ are the eccentricities and $i_1, i_2$ are the inclinations of the two ellipses respectively.  From their definitions, we see that these coordinates are well-defined only when neither of the ellipses is circular, horizontal or rectilinear. We refer to \cite{Lecons}, \cite{DelaunayPoincare} or \cite[Appendix A]{FejozHabilitation} for more detailed discussions of Delaunay coordinates.

In these coordinates, the Keplerian part $F_{Kep}$ is in the action-angle form
$$F_{Kep}=-\dfrac{\mu_{1}^{3} M_{1}^{2}}{2 L_{1}^{2}}-\dfrac{\mu_{2}^{3} M_{2}^{2}}{2 L_{2}^{2}}.$$
The proper degeneracy of the Kepler problem can be seen by the fact that $F_{Kep}$ depends only on 2 of the action variables out of 6. As a result, in order to study the dynamics of $F$, it is crucial to look at higher order effects arising from $F_{pert}$. 

\subsection{Reduction of the SO(3)-Symmetry: Jacobi's Elimination of the Nodes}
The group SO(3) acts on $\Pi$ by simultaneously rotating the positions $Q_{1}, Q_{2}$ and the momenta $P_{1}, P_{2}$ around the origin. This action is Hamiltonian for the standard symplectic form on $\Pi$ and it leaves the Hamiltonian $F$ invariant. Its moment map is the total angular momentum $\vec{C}=\vec{C_1}+\vec{C_2}$, in which $\vec{C}_{1}:=Q_1 \times P_1$ and $\vec{C}_{2}:=Q_2 \times P_2$.  To reduce $F$ by this SO(3)-symmetry, we fix $\vec{C}$ to a regular value (\emph{i.e.} $\vec{C} \neq \vec{0}$) and then reduce the system from the SO(2)-symmetry around $\vec{C}$. Finally, we obtain from $F$ a Hamiltonian system with 4 degrees of freedom.

\subsubsection{Jacobi's elimination of the nodes}
The plane perpendicular to $\vec{C}$ is invariant. It is called the \emph{Laplace plane}. We choose it to be the reference plane.

Since the angular momenta $\vec{C}_1$, $\vec{C}_2$ of the two Keplerian motions and the total angular momentum $\vec{C}=\vec{C}_1+\vec{C}_2$ must lie in the same plane, the node lines of the orbital planes of the two ellipses in the Laplace plane must coincide (\emph{i.e.} $h_{1}=h_{2}+\pi$).  Therefore, by fixing the Laplace plane as the reference plane, we can express $H_1,H_2$ as functions of $G_1$, $G_2$ and $C:=\|\vec{C}\|$:
$$H_{1}=\dfrac{C^{2}+G_{1}^{2}-G_{2}^{2}}{2 C}, H_{2}=\dfrac{C^{2}+G_{2}^{2}-G_{1}^{2}}{2 C};$$
since $\vec{C}$ is vertical, $d H_{1} \wedge d h_{1}+ d H_{2} \wedge d h_{2}=d C \wedge d h_{1}$.

We can then reduce the system by the SO(2)-symmetry around the direction of $\vec{C}$. The number of degrees of freedom of the system is then reduced from 6 to 4. 

\section{The Averaging Procedure} \label{SecularAndSecularIntegrableSystems}
In the lunar case, the two Keplerian frequencies do not appear at the same magnitude of the small parameter $\alpha$, which enables us to build normal forms up to any order, without necessarily considering the interaction between the two Keplerian frequencies by an asynchronous elimination procedure that we are going to describe, which was carried out by Jefferys and Moser in \cite{JefferysMoser}, with an alternative presented in \cite{QuasiMotionPlanar} in which the terminology \emph{asynchronous region} is coined.  In order to build integrable approximating systems, we shall further average over $g_{2}$ to obtain the secular-integrable systems by an additional single frequency averaging.

\subsection{Asynchronous Region}
We fix the masses $m_{0}, m_{1}, m_{2}$ arbitrarily, and suppose that the eccentricities $e_{1}$ and $e_2$ are bounded away from $0, 1$, so there exist positive real numbers $e_{1}^{\vee}, e_{1}^{\wedge}, e_{2}^{\vee}, e_{2}^{\wedge}$, such that
$$0<e_{1}^{\vee} < e_{1} < e_{1}^{\wedge}<1,\quad 0<e_{2}^{\vee} < e_{2} < e_{2}^{\wedge}<1.$$
Recall that the small parameter $\alpha=\dfrac{a_{1}}{a_{2}}$ is the ratio of the semi major axes. We suppose that
$$\alpha< \alpha^{\wedge}:=\min\{\dfrac{1-e_{2}^{\wedge}}{80}, \dfrac{1-e_{2}^{\wedge}}{2 \sigma_{0}}, \dfrac{1-e_{2}^{\wedge}}{2 \sigma_{1}}\},$$ 
in which $\dfrac{1}{\sigma_0}=1+\dfrac{m_1}{m_0},  \dfrac{1}{\sigma_1}=1+\dfrac{m_0}{m_1}$ (see Appendix \ref{section:estimates} for the choice of $\alpha^{\wedge}$). In particular, 
$$\max\{\sigma_{0}, \sigma_{1}\} \alpha \dfrac{1+e_{1}}{1-e_{2}} < 1,$$
\emph{i.e.}, the two ellipses are always bounded away from each other for all the time.

Without loss of generality, we fix two real numbers $a_{1}^{\wedge}>a_{1}^{\vee}>0$, such that the relation $a_{1}^{\vee} < a_{1} < a_{1}^{\wedge}$ holds for all time.

The subset of the phase space $\Pi$ in which Delaunay coordinates for both ellipses are regular coordinates, and satisfy these restrictions is denoted by $\mathcal{P}^*$ (what could be called the \emph{asynchronous region}): it can thus be regarded (by Delaunay coordinates) as a subset\footnote{The condition defining $\mathcal{P}^*$ can be replaced by other conditions, e.g. by asking that the Deprit coordinates to be regular.} of $\T^6 \times \R^{6}$. The function $F_{pert}$ can thus be regarded as an analytic function on $\mathcal{P}^* \subset \T^6 \times \R^{6}$.

Let $\nu_1, \nu_2$ denote the two Keplerian frequencies: $\nu_i=\dfrac{\partial F_{Kep}}{\partial L_i}=\sqrt{\dfrac{M_i}{a_i^3}}, i=1,2$. 

Let $T_{\C}=\C^6/\Z^6 \times \C^6$ and $T_s$ be the $s$-neighborhood of $\T^6 \times \R^6:=\R^{6}/\Z^{6} \times \R^{6}$ in $T_{\C}$. Let $T_{\textbf{A},s}$ be the $s$-neighborhood of a set $\textbf{A} \subset \T^6 \times \R^6$ in $T_s$. 
The complex modulus of a transformation is the maximum of the complex moduli of its components. We use $| \cdot |$ to denote the modulus of either a function or a transformation. 

Lemma \ref{Appendix A: lem 3 ddd} states that there exists some small real number $s>0$, such that in $T_{\mathcal{P}^{*},s}$, $|F_{pert} | \le \hbox{Cst}\, |\alpha|^{3} $, in which the constant $\hbox{Cst}$ is independent of $\alpha$. 
\subsection{Asynchronous Elimination of the Fast Angles}\label{Subsection: Elimination of the Fast Angles} 

\begin{prop}\label{averaging} For any (fixed) $n \in \N$ , there exist an analytic Hamiltonian $F^n: \mathcal{P}^{*} \to \R$ independent of the fast angles $l_1,l_2$, and an analytic symplectomorphism $\phi^n: \mathcal{P}^{*} \supset \tilde{\mathcal{P}} \to \phi^n (\tilde{\mathcal{P}})$, $|\alpha|^{\frac{3}{2}}$-close to the identity, such that
             $$|F \circ \phi^n- F^n| \le C_{0}\, |\alpha |^{\frac{3(n+2)}{2}}$$
on $T_{\tilde{\mathcal{P}}, s^{''}}$ for some open set $\tilde{\mathcal{P}} \subset \mathcal{P}^{*}$, and some real number $s^{''}$ with $0< s^{''} < s$. Moreover, the relative measure of $\tilde{\mathcal{P}}$ in $\mathcal{P}^{*}$ tends to 1 when $\alpha$ tends to $0$. 
\end{prop}

\begin{proof}
The strategy is to first eliminate $l_1$ up to sufficiently large order, and then eliminate $l_2$ to the desired order.
We describe the first step of eliminating $l_{1}$. 

To eliminate the angle $l_1$ in the perturbing function $F_{pert}$, we look for an auxiliary analytic Hamiltonian $\hat{H}$. We denote its Hamiltonian vector field by $X_{\hat{H}}$ and its flow by $\phi_t$. The symplectic coordinate transformation that we are looking for is given by the time-1 map $\phi_1(:= \phi_t|_{t=1})$ of $X_{\hat{H}}$.

Define the first order complementary part $F^1_{comp,1}$ by the equation
$$\phi_1^*F=F_{Kep}+(F_{pert}+X_{\hat{H}} \cdot F_{Kep})+F^1_{comp,1},$$
in which $X_{\hat{H}}$ is seen as a derivation operator. Let
$$\langle F_{pert}\rangle_1 = \dfrac{1}{2 \pi} \int_{0}^{2 \pi} F_{pert} \,d l_1$$
be the average of $F_{pert}$ over $l_1$, and $\widetilde{F}_{pert,1}=F_{pert}-\langle F_{pert}\rangle_1$ be its zero-average part.

As the two Keplerian frequencies do not appear at the same magnitude of $\alpha$, we do not need to ask $\hat{H}$ to solve the (standard) cohomological equation:
$$\nu_1 \partial_{l_1} \hat{H}+ \nu_2 \partial_{l_2} \hat{H}=\widetilde{F}_{pert,1};$$
instead, we just need $\hat{H}$ to solve the perturbed cohomological equation
$$\nu_1 \partial_{l_1} \hat{H}=\widetilde{F}_{pert,1}.$$
We thus set
$$\hat{H} (l_2)= \dfrac{1}{\nu_1}  \int_0^{l_1}  \widetilde{F}_{pert,1}\, d l_1$$
as long as $\nu_1 \neq 0$, which is indeed satisfied for any Keplerian frequency (of an elliptic motion). This amounts to proceed with a single frequency elimination for $l_{1}$.
We have 
 $$|\hat{H}| \le \hbox{Cst } |\alpha|^{3} \, \hbox{   in   } T_{\mathcal{P}^{*}, s} . $$
We obtain by Cauchy inequality that in $T_{\mathcal{P}^{*},s-s_0}$, $|X_{\hat{H}}| \le \hbox{Cst } |\hat{H} | \le \hbox{Cst } |\alpha|^{3}$ for some $0<s_0<s/2$. Shrinking from $T_{\mathcal{P}^{*},s- s_0}$ to $T_{\mathcal{P}^{**},s- s_0-s_1}$, where $\mathcal{P}^{**}$ is an open subset of $\mathcal{P}^{*}$, so that $\phi_1(T_{\mathcal{P}^{**},s- s_0-s_1}) \subset T_{\mathcal{P}^{*},s-s_0}$, with $s- s_0-s_1>0$. The time-1 map $\phi_1$ of  $X_H$ thus satisfies $|\phi_1-Id | \le \hbox{Cst } |\alpha |^{3}$ in $T_{\mathcal{P}^{**},s- s_0- s_1}$. The function $\phi_1^* F$ is analytic in $T_{\mathcal{P}^{**},s- s_0-s_1}$.

 Now $F$ is conjugate to
$$\phi^*_1 F =F_{Kep}+\langle F_{pert} \rangle_1 + F^1_{comp,1},$$ and $|F^1_{comp,1}|$ is of order $O(\alpha^{\frac{9}{2}})$: indeed, analogous to \cite{QuasiMotionPlanar}, the complementary part
 $$F_{comp,1}^1 = \int_0^1 (1-t) \phi_t^*(X_{\hat{H}}^2 \cdot F_{Kep})dt+ \int_0^1 \phi_t^*(X_{\hat{H}} \cdot F_{pert}) dt  - \nu_2 \dfrac{\partial \hat{H}}{\partial l_2}$$

 satisfies

 $$|F_{comp,1}^1| \le \hbox{Cst } |X_{\hat{H}}| (|\widetilde{F}_{pert,1}|+|F_{pert}|) + \nu_2 | \hat{H}| \le \hbox{Cst } | \alpha |^{\frac{9}{2}}.$$

The first order averaging with respect to $l_1$ is then accomplished.

One proceeds analogously and eliminates the dependence of the Hamiltonian of $l_1$ up to order $O(\alpha^{\frac{3(n+2)}{2}})$ for any chosen $n \in \Z_{+}$. The Hamiltonian $F$ is then analytically conjugate to

$$F_{Kep}+\langle F_{pert} \rangle_1 + \langle F^{1}_{comp,1} \rangle_{1}+ \cdots +\langle F^{1}_{comp,n-1}\rangle_{1} + F^{1}_{comp,n}, $$

in which the expression $F_{Kep}+\langle F_{pert} \rangle_1 + \langle F^{1}_{comp,1} \rangle_{1}+ \cdots + \langle F^{1}_{comp,n-1}\rangle_{1}$ is independent of $l_1$, and $F^{1}_{comp,n}$ is of order $O \left(\alpha^{\frac{3(n+2)}{2}}\right)$.

After this, we proceed by eliminating $l_2$ from 
$$F_{Kep}+\langle F_{pert} \rangle_1 + \langle F^{1}_{comp,1} \rangle_{1}+ \cdots + \langle F^{1}_{comp,n-1}\rangle_{1},$$
which is again a single frequency averaging and it can be carried out as long as $\nu_2 \neq 0$, which is always true under our hypothesis.

The Hamiltonian generating the transformation for the first step of averaging over $l_2$ is
$$\dfrac{1}{\nu_2} \int_0^{l_2} (\langle F_{pert} \rangle_1- \langle F_{pert} \rangle) d l_2 \le \hbox{Cst } |\alpha |^{\frac{3}{2}}.$$
The other steps are similar to the first step of eliminating $l_1$.
By eliminating $l_2$, the Hamiltonian $F$ is conjugate to
$$F_{Kep}+\langle F_{pert} \rangle + \langle F_{comp,1} \rangle+ \cdots +\langle F_{comp,n-1}\rangle+F_{comp,n},$$
in which the (first order) \emph{secular system}
$$F_{sec}^{1}=\langle F_{pert} \rangle := \dfrac{1}{4 \pi^{2}} \int_{0}^{2 \pi} \int_{0}^{2 \pi} F_{pert} d l_{1} d l_{2}$$
and the \emph{$n$-th order secular system}
$$F_{sec}^{n}:=\langle F_{pert} \rangle + \langle F_{comp,1} \rangle+ \cdots +\langle F_{comp,n-1}\rangle$$\label{Not: sec system}
is independent of $l_1,l_2$, with 
$$\langle F_{comp,i} \rangle =O(\alpha^{\frac{3(i+2)}{2}}),\,F_{comp,n}=O(\alpha^{\frac{3(n+2)}{2}})$$ in $T_{\tilde{\mathcal{P}}, s^{''}}$ for some open subset $\tilde{\mathcal{P}} \subset \mathcal{P}^{*}$ and some $0<s^{''}<s$ both of which are obtained by finite steps of constructions analogously to that we have described for the first step elimination of $l_{1}$. In particular, the set $\tilde{\mathcal{P}}$ is obtained by shrinking $\mathcal{P}^{*}$ from its boundary by a distance of  $O(\alpha^\frac{3}{2})$. We may thus set 
$$F^{n}:=F_{Kep}+F_{sec}^{n}.$$
\end{proof}

The function $F_{sec}^{n}$ is defined on a subset of the phase space $\Pi$ and does not depend on the fast Keplerian angles. After fixing $L_{1}$ and $L_{2}$ and reducing by the Keplerian $\T^{2}$-symmetry, the reduced function is defined on a subset of the \emph{secular space}, or the space of pairs of ellipses. By the construction presented in \cite{Albouy}, the secular space is seen to be homeomorphic to $(S^{2} \times S^{2})^{2}$.
We keep the same notation $F_{sec}^{n}$ for the resulting function.

\subsection{Secular-integrable Systems} \label{Subsection: Sec-int Sys}
Unlike the integrable planar secular systems, the spatial secular systems $F_{sec}^n$ remains to have 2 degrees of freedom after being reduced by the SO(3)-symmetry, and therefore they are \emph{a priori} not integrable. As a result, in contrast to the planar case, they cannot directly serve as an ``integrable approximating system'' for our study. 

{In $\mathcal{P}^{*}$, the function $F_{sec}^{1}$ is of order $O(\alpha^{3})$, and the functions $F_{comp, n}, n \ge 2$ are of order $O(\alpha^{\frac{9}{2}})$. }
{We express $F_{sec}^{n}$ as
$$F_{sec}^{n}(a_{1}, \alpha, e_{1}, e_{2}, g_{1}, g_{2}, h_{1}, h_{2}, i_{1}, i_{2}),$$
and expand it in powers of $\alpha$:
$$ F_{sec}^n=\sum_{i=0}^{\infty} F_{sec}^{n,i} \alpha^{i+1}=F_{sec}^{n,0} \alpha+F_{sec}^{n,1} \alpha^{2}+\cdots.$$
As a consequence of Lemma \ref{Legendre Expansion}, we see that
$$\forall n \in \mathbb{N}_{+}, F_{sec}^{n,i}=0, \qquad i=0,1.$$
Moreover, since $F_{comp}^{n}, n \ge 1$ is of order $O(\alpha^{\frac{9}{2}})$, we have $$F_{sec}^{n}-F_{sec}^{1}=O(\alpha^{\frac{9}{2}}),$$
in particular
$$F_{sec}^{n,2}=F_{sec}^{1,2}, \qquad \forall n=1,2,3,\cdots.
$$
}

As noticed by Harrington in \cite{Harrington1968}\footnote{For the (inner) restricted spatial three-body problem, the integrability of the quadrupolar system has been discovered in 1961 by Lidov  \cite{Lidov1961} (see also Lidov \cite{Lidov1962}, Kozai \cite{Kozai1962}). Its link with the non-restricted quadrupolar system has been discussed in \cite{LidovZiglin}. }, the term $F_{sec}^{1,2}$ is independent of $g_2$, thus $G_2$ is an additional first integral of the system $F_{sec}^{1,2}$. The system $F_{sec}^{1,2}$ can then be reduced to one degree of freedom after reduction of the symmetries, hence it is integrable. We call $F_{quad}:=F_{sec}^{1,2}$
 the \emph{quadrupolar system}.

The integrability of the quadrupolar Hamiltonian is, in Lidov and Ziglin's words, a ``happy coincidence'': it is due to the particular form of $F_{pert}$. Indeed, if one goes to even higher order expansions in powers of $\alpha$, then in general the truncated Hamiltonian will no longer be independent of $g_2$ (c.f. \cite{LaskarBoue}).

To have better control of the perturbation so as to apply KAM theorems, we need to build higher order integrable approximations by eliminating $g_2$ in the secular systems $F_{sec}^{n}$. This is a single frequency elimination procedure and can be carried out everywhere as long as the frequency $\nu_{quad, 2}$ of $g_2$ in $F_{quad}$ is not zero. 

Since the analytic function $F_{quad}$ depends non trivially on $G_2$ (See Section \ref{QuadrupolarDynamics}), for any $\varepsilon$ small enough, we have $|\nu_{quad, 2}| > \varepsilon$ on an  open subset $\check{\mathcal{P}}$ of $\mathcal{P}^{*}$ and locally relative measure of $\check{\mathcal{P}}$ in $\mathcal{P}^{*}$ tends to 1 when $\varepsilon$ tends to 0. For any fixed $\varepsilon$, analogous to Subsection \ref{Subsection: Elimination of the Fast Angles}, for small enough $\alpha$, there exists an open subset $\hat{\mathcal{P}}$ in $\check{\mathcal{P}}$ with local relative measure in $\check{\mathcal{P}}$ tending to 1 when $\alpha$ tends to 0, such that on $\hat{\mathcal{P}}$ we can conjugate our system up to small terms of higher orders to the normal form that one gets by the standard elimination procedure (c.f. \cite{ArnoldGeometrical}) to eliminate $g_{2}$.

 More precisely, after fixing the Laplace plane as the reference plane, as the elimination of $l_{2}$ in the proof of Proposition \ref{averaging}, for the first step of elimination, we eliminate the angle $g_{2}$ in $F_{Kep}+\alpha^{3}(F_{quad}+\alpha F_{sec}^{1,3})$ by a symplectic transformation $\psi^{3}$ close to identity,  which is the time-1 map of the Hamiltonian
$$
\dfrac{\alpha}{\nu_{g_{2}}}\left(\int_{0}^{g_{2}} \Bigl(F_{sec}^{1,3}-\dfrac{1}{2 \pi} \int_{0}^{2 \pi} F_{sec}^{1,3} d g_{2}\Bigr) d g_{2}\right).
$$

We proceed analogously for higher order eliminations. We denote by $\psi^{n'}: \hat{\mathcal{P}} \to \psi^{n'} (\hat{\mathcal{P}})$ the corresponding symplectic transformation, so that 
$$\psi^{n'*}F_{sec}^{n}=\alpha^3 F_{quad}+\alpha^4 \widetilde{ F_{sec}^{n,3} }+ \cdots + \alpha^{n'} \widetilde{ F_{sec}^{n,n'} }+ F_{secpert}^{n'+1},$$
in which $F_{secpert}^{n'+1}= O(\alpha^{n'+2})$ and $\widetilde{F_{sec}^{n,i}}, i=1, 2,\cdots$ are independent of $g_{2}$. 

Let 
$$\overline{F_{sec}^{n,n'}} = \alpha^3 F_{quad}+\alpha^4 \widetilde{F_{sec}^{n,3}}+ \cdots + \alpha^{n'} \widetilde{F_{sec}^{n,n'}};$$
we call it the $(n,n')$-th order \emph{secular-integrable system}. We have 
$$ \psi^{n'*} \phi^{n*} F=F_{Kep}+\overline{F_{sec}^{n,n'}}+F_{secpert}^{n'+1}+F_{comp}^n.$$
For $\alpha$ small enough, the last two terms can be made arbitrarily small by choosing $n,n'$ large enough. 

\section{The Quadrupolar Dynamics} \label{QuadrupolarDynamics}

The secular-integrable systems $\overline{F_{sec}^{n,n'}}$ are $O(\alpha^4)$ perturbations\footnote{Actually $\overline{F_{sec}^{n,3}}=0$ but $\overline{F_{sec}^{n,4}} \neq 0$, therefore $\overline{F_{sec}^{n,n'}}-\alpha^3 F_{quad}$ is of order $O(\alpha^{5})$.} of $\alpha^3 F_{quad}$, therefore for $\alpha$ small, the key to understand the dynamics of $\overline{F_{sec}^{n,n'}}$ is to understand the dynamics of $F_{quad}$ (seen as a function defined on a subset of the secular space). 
In this section, we shall reproduce some of the study of Lidov-Ziglin in \cite{LidovZiglin}. 

After Jacobi's elimination of the nodes, the quadrupolar Hamiltonian takes the form
\small
\begin{align*}
 F_{quad}&=-\dfrac{\mu_{quad} L_{2}^{3}}{8 a_{1} G_2^3} \left\{3\dfrac{G_1^2}{L_1^2} \Bigl[1+\dfrac{(C^2-G_1^2-G_2^2)^2}{4 G_1^2 G_2^2}\Bigr] \right.\\
 &\left.\phantom{aaaaaaaaaaaa}+ 15 \Bigl(1-\dfrac{G_1^2}{L_1^2}\Bigr)\Bigl[\cos^2{g_1}+\sin^2{g_1} \dfrac{(C^2-G_1^2-G_2^2)^2}{4 G_1^2 G_2^2}\Bigr] -6\Bigl(1-\dfrac{G_1^2}{L_1^2}\Bigr)-4\right\},
\end{align*}
\normalsize
in which $\mu_{quad}=\dfrac{m_{0} m_{1} m_{2}}{m_{0}+m_{1}}$. 

{\bf Notations}: We separate the variables of the system and the parameters by a semicolon so as to make the difference between different reduced systems more apparent: 
The functions $L_{1}$, $L_{2}$, $C$ and $G_{2}$ are first integrals of $F_{quad}(G_{1}, g_1, C, G_{2}, L_{1}, L_{2})$. If we fix these first integrals and reduce the system by the conjugate $\T^{4}$-symmetry, then $C$ and $G_{2}$ becomes parameters of the reduced system as well. The resulting system is thus written as $F_{quad}(G_{1}, g_1; C, G_{2}, L_{1}, L_{2})$.
  
By applying the triangular inequality to the vectors $\vec{C}$, $\vec{C}_{1}$, $\vec{C}_{2}$, we see that the parameters $L_{1}$, $C$ and $G_2$ must satisfy the condition 
$$|C-G_2| \leq L_1,$$ 
which defines the region of admissible parameters in the $(C, G_2)$-parameter space for fixed $L_{1}$. By triangular inequality and definition of $G_{1}$, when $C$ and $G_2$ are fixed, the quantity $|G_1|$ belongs to the interval $[G_{1,min}, G_{1,max}]$, for $G_{1,min}:=|C-G_2|, G_{1,max}:=\min \{L_1, C+G_2\}$. 

After proper blow-up of the secular space (by adding artificial pericentre/node directions to circular/coplanar ellipses), we may still use 
$(G_{1}, g_1, G_{2}, \bar{g}_{2})$ to characterize circular inner or outer ellipses or coplanar pairs of ellipses. In this section, as implicitly in Lidov-Ziglin \cite{LidovZiglin}, we retain this convention unless otherwise stated. Note that the reduction procedure of the $\hbox{SO(2)}$-symmetry around $\vec{C}$ for coplanar pairs of ellipses after the blow-up procedure, however, does not lead to an effective reduction procedure in the secular space. This will cause no problem for our study.

From its explicit expression, when $C \neq G_{2}$, we see that the Hamiltonian $F_{quad}(G_{1}, g_1; C, G_{2}, L_{1}, L_{2})$ is nevertheless regular for all $0<G_{1}< L_{1}$. This phenomenon comes from the expression of $\cos{(i_{1}-i_{2})}$ as a function of $C, G_{1}, G_{2}$, and thus also holds for any $\overline{F}_{sec}^{n,n'}$.\footnote{Each $\overline{F_{sec}^{n,n'}}$ depends polynomially on $\cos(i_{1}-i_{2})$ (through Legendre polynomials), therefore it remains analytic in $G_{1}$ for $0<G_{1} <L_{1}$ if we substitute $\cos(i_{1}-i_{2})$ by $\frac{C^{2}-G_{1}^{2}-G_{2}^{2}}{2 G_{1} G_{2}}$.} For $G_{1} < G_{1, min}$, the dynamics determined by the above expression of $F_{quad}$ is irrelevant to the real dynamics, but the fact that the expression of $F_{quad}$ is analytic in $G_{1}$ for all $0<G_{1} <L_{1}$ enable us to develop $F_{quad}$ into Taylor series of $G_{1}$ at $\{G_{1}=G_{1,min}\}$ for $G_{1,min}>0$. In Appendix \ref{non-degeneracySecularFrequencyMap}, this allows us to show the existence of torsion for those quadrupolar invariant tori near $\{G_{1}=G_{1,min}\}$ with some simple calculations.

Now we may fix $C$ and $G_{2}$ and reduce the system to one degree of freedom. When $C \neq G_{2}$, the (physically relevant) reduced quadrupolar dynamics lies in the cylinder defined by the condition $$G_{1,min} \le G_{1} \le G_{1,max}.$$ 
As is shown by Lidov-Ziglin, for fixed $L_{1}$ and $L_{2}$, in different regions of the $(C, G_{2})$-parameter space, the phase portraits in the $(G_1, g_1)$- plane have periodic orbits, finitely many singularities
and separatrices; the first two kinds give rise to invariant 2-tori and periodic orbits of the reduced system of $F_{quad}(G_{1}, g_1, G_{2}; C, L_{1}, L_{2})$ by the $\hbox{SO(3)}$-symmetry. 

The quadrupolar phase portraits in the $(G_{1}, g_1)$-space are invariant under the translations 
$$(g_1, G_{1}) \to (g_1 + n \, \pi, G_{1}), n \in \Z,$$
and the reflections
$$(g_1, G_{1}) \to (\pi-g_1 , G_{1}).$$
Therefore, without loss of generality, we can identify points obtained by reflexions and translations. In particularly, we shall make this identification for the singularities. 

When $C \neq G_2$, the dynamics of $F_{quad}$ can be easily deduced from \cite{LidovZiglin} by using the relations ($\bbespilon,\bbomega$\label{Not: quadru variables LZ} denote respectively the symbols $\varepsilon,\omega$ in \cite{LidovZiglin}) 
$$\bbespilon=\dfrac{G_1^2}{L_1^2},\quad \bbomega=g_2.$$
According to different choices of parameters,  we list different quadrupolar phase portraits  in the following:

\begin{enumerate}
\item $G_{2} < C, 3 G_{2}^{2} + C^{2} < L_{1}^{2}.$\\
In this case, there exists an elliptical singularity 
$$B: (g_1\equiv\dfrac{\pi}{2} \, (mod \, \pi), G_{1}=G_{1,B}),$$
where $G_{1,B}$ is determined by the equation
$$\dfrac{G_{1,B}^{6}}{L_{1}^{6}} - \left(\dfrac{G_{2}^{2} + 2 C^{2}}{2 L_{1}^{2}}+\dfrac{5}{8}\right) \dfrac{G_{1,B}^{4}}{L_{1}^{4}} + \dfrac{5 (C^{2}-G_{2}^{2})^{2}}{8 L_{1}^{4}} =0.$$
    
There also exists a hyperbolic singularity 
$$A: \left (g_1 \equiv 0  \, (mod \, \pi ), G_{1}=\sqrt{3 G_{2}^{2} + C^{2}} \right).$$

\item $G_{2} + C < L_{1}, 0 < (G_{2}-C)(G_{2}+C)^{2} < 5 C (L_{1}^{2}-(C+G_{2})^{2})$ \\or \\     $G_{2} + C > L_{1}, 0 < 2 L_{1}^{2} (3 G_{2}^{2} + C^{2} - L_{1}^{2}) < 5  (4 L_{1}^{2} G_{2}^{2}-(C^{2}-G_{2}^{2}-L_{1}^{2})^{2})$.
     
         In this case, there exist two singularities: the elliptic singularity $B$, and a hyperbolic singularity $E$:
         $$E: (g_1 \equiv \arcsin \sqrt{\dfrac{(G_{2}-C)(G_{2}+C)^{2}}{5 C (L_{1}^{2}-(G_{2}+C)^{2})}}  (mod \, \pi ), G_{1}=G_{1,max})$$
         if $C+G_{2}<L_{1}$, and
          $$E: (g_1\equiv \arcsin \sqrt{\dfrac{2 L_{1}^{2} (3 G_{2}^{2} + C^{2} - L_{1}^{2})}{5  (4 L_{1}^{2} G_{2}^{2}-(C^{2}-G_{2}^{2}-L_{1}^{2})^{2})}}  (mod \, \pi ), G_{1}=G_{1,max})$$
          if $C+G_{2} > L_{1}$.

\item $(C-G_{2})^{2}<\dfrac{2}{3} (\dfrac{G_{2}^{2}}{2} + C^{2} + \dfrac{5 L_{1}^{2}}{8}) < \min \{L_{1}^{2}, (C+G_{2})^{2}\}$\label{Condition (3)} \\
          $L_{1}^{2} (C^{2}+G_{2}^{2})^{2} < \dfrac{32}{135} (\dfrac{G_{2}^{2}}{2} + C^{2} + \dfrac{5 L_{1}^{2}}{8})^{3}$ \\
          $5 C (L_{1}^{2}-(C+G_{2})^{2}) < (G_{2}-C)(G_{2}+C)^{2}$, if $C+G_{2} < 1$ and\\
          $5  (4 L_{1}^{2} G_{2}^{2}-(C^{2}-G_{2}^{2}-L_{1}^{2})^{2}) < 2 L_{1}^{2} (3 G_{2}^{2} + C^{2} - L_{1}^{2})$, if $C+G_{2} > 1$. \\
           In this case, there exists an elliptic singularity $B$ and a hyperbolic singularity $A'$ on the line defined by $g_1 \equiv \dfrac{\pi}{2} \,(mod \, \pi)$. The ordinate of $A'$ is determined by the same equation that defines the ordinate of $B$ in the case (1).
           
\item The border cases of the above-listed choices of parameters.
         For such parameters, the corresponding phase portraits can be easily deduced by some limiting procedures. We shall not need them in this study. 
                                
\item There are no singularities for other choice of parameters.

\end{enumerate}

In the case $C=G_{2}$, the invariant curves in the corresponding one degree of freedom system cannot avoid passing degenerate inner ellipses, for which the angle $l_{1}$ and thus the averaging procedure are not well-defined. Moreover, the corresponding inner Keplerian dynamics cannot avoid double collisions. The persistence of the corresponding invariant tori in $F$ necessarily requires the regularization of the inner double collisions and study the corresponding secular dynamics. We avoid analyzing this case in this article, and refer to \cite{AlmostCollision} for the precise treatment of the analysis of the quadrupolar system near $C=G_{2}$ and the related perturbative study.

Figures \ref{LZ Parameter Figure} and \ref{Fig: quadrupolar Phase LZ} are the parameter space and phase portraits of the quadrupolar system respectively, which are essentially those of \cite{LidovZiglin}.

\begin{figure} 
\centering
\includegraphics[width=3.5in]{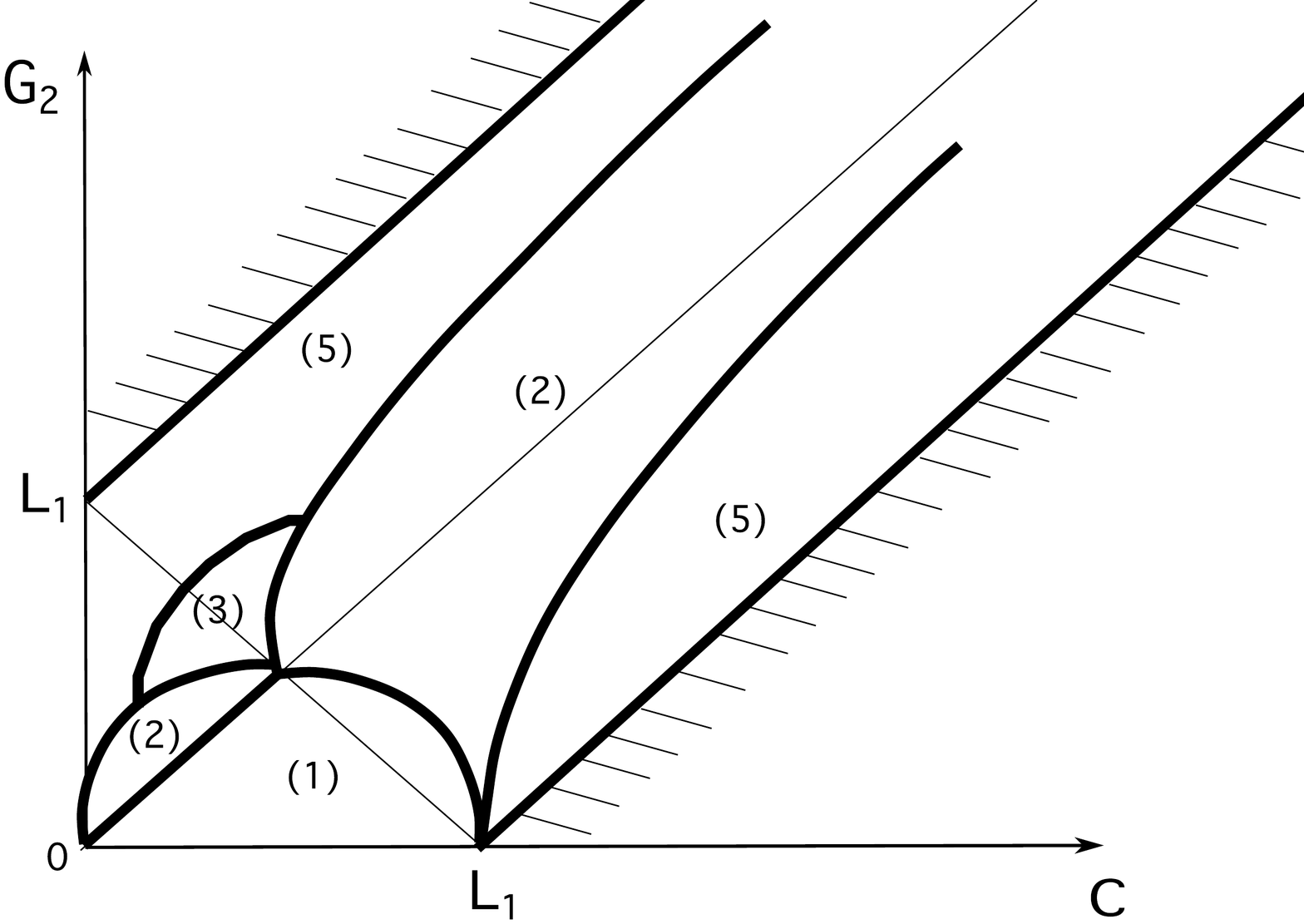} 
\caption{The parameter space of the quadrupolar system} \label{LZ Parameter Figure}
\end{figure}

\begin{figure}
\centering
\includegraphics[width=4.5in]{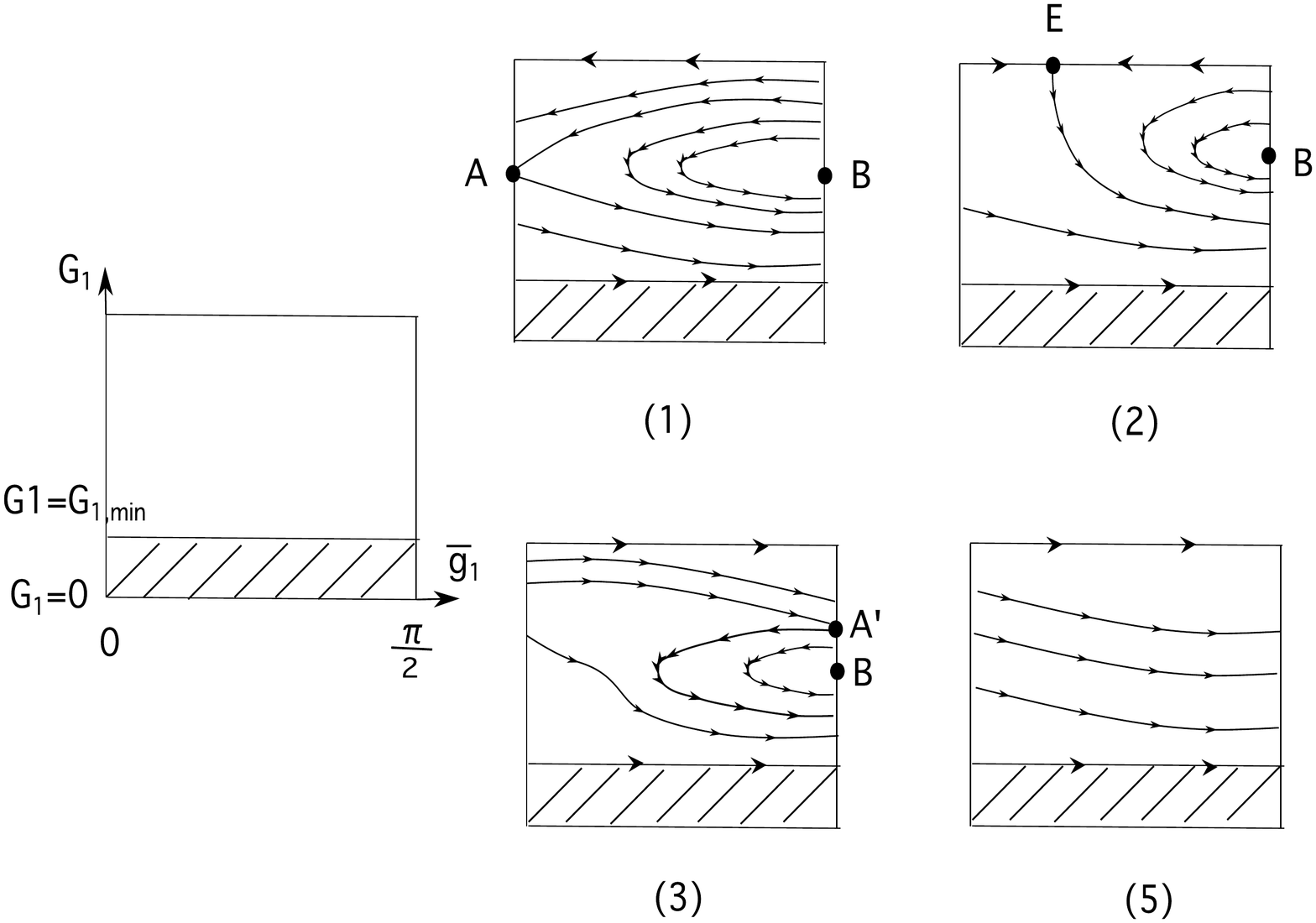} 
\caption{The phase portraits of the quadrupolar system for $C \neq C_{2}$.} \label{Fig: quadrupolar Phase LZ}
\end{figure}

By construction, for each positive integer pair $(n, n')$, the higher order secular-integrable systems $\overline{F_{sec}^{n,n'}}$ also has first integrals $C$ and $G_{2}$. As for $F_{quad}$, when $C \neq G_{2}$, the inner eccentricity $e_{1}$ is bounded away from $1$.  After fixing $C$ and $G_{2}$ and reducing $\overline{F_{sec}^{n,n'}}$ by the SO(2)-actions of their conjugate angles, the reduced dynamics of $\overline{F_{sec}^{n,n'}}$ is defined in the same space as that of $F_{quad}$ for $C \neq G_{2}$. By analyticity of $F_{quad}$, we will show that for a dense open set of the parameter space, the singularities $A, B, A', E$ are of Morse type (Proposition \ref{Quadrupolar Singularity non-degeneracy}). Being Morse singularities, they persist under small perturbations and thus serve as singularities for $\overline{F_{sec}^{n,n'}}$ for small enough $\alpha$. The phase portraits of $\overline{F_{sec}^{n, n'}}$ are just small perturbations of (and orbitally conjugate to) that of the quadrupolar system $F_{quad}$. 

\section{The KAM Theorem and Applications}
In this section, we first give an analytic version of a powerful ``hypothetical
conjugacy'' theorem (c.f. \cite{FejozStability}), which does not depend on any non-degeneracy condition. We then
discuss the classical (strong) iso-chronous non-degeneracy condition which
guarantees the existence of KAM tori. Finally, a theorem of J.P{\"o}schel is presented, which shows the existence of families of periodic solutions accumulating KAM tori. 

\subsection{Hypothetical Conjugacy Theorem}

For $p \ge 1$ and $q \ge 0$, consider the phase space $\R^{p} \times
\T^p \times \R^q \times \R^q = \{(I,\theta,x,y)\}$ endowed with the
standard symplectic form $d I \wedge d \theta + d x \wedge d y$. All
mappings are assumed to be analytic except when explicitly mentioned
otherwise.

Let $\delta > 0$, $q'\in\{0,...,q\}$, $q''=q-q'$, $\varpi \in \R^p$,
and $\beta \in \R^q$. Let $B_{\delta}^{p+2 q}$ be the $(p+2 q)$-dimensional closed ball with radius $\delta$ centered at the origin in $\R^{p+2q}$, and $N_{\varpi,\beta}= N_{\varpi,\beta}
(\delta,q')$ be the space of Hamiltonians $N \in C^{\omega}(\T^p
\times B_{\delta}^{p+2q},\R)$ of the form
$$N= c + \langle \varpi,I \rangle + \sum_{j=1}^{q'}
\beta_j (x_j^2+ y_j^2) + \sum_{j=q'+1}^{q} \beta_j (x_j^2 -
y_j^2)+\langle A_1(\theta), I \otimes I \rangle + \langle A_2(\theta),
I \otimes Z \rangle + O_3(I,Z),
$$ 
with $c\in \R$, $A_1 \in C^{\omega}(\T^p, \R^{p} \otimes \R^p), A_2
\in C^{\omega}(\T^p,\R^p \otimes \R^{2q})$ and $Z = (x,y)$. The
isotropic torus $\T^p \times \{0\} \times \{0\}$ is an invariant
$\varpi$-quasi-periodic torus of $N$, and its normal dynamics is
elliptic, hyperbolic, or a mixture of both types, with Floquet
exponents $\beta$. The definitions of tensor operations can be found in e.g. \cite[p.62]{FejozStability}.

Let $\bar{\gamma} > 0$ and $\bar{\tau} > p-1$, $|\cdot|$ be the $\ell^2$-norm on
$\Z^{p}$. Let
$HD_{\bar\gamma,\bar\tau}=HD_{\bar{\gamma},\bar{\tau}}(p,q',q'')$ be the set of vectors $(\varpi,\beta)$ satisfying the
following homogeneous Diophantine conditions:
$$|k \cdot \varpi + l' \cdot \beta'| \geq \bar\gamma
(|k|^{\bar\tau}+1)^{-1}$$ for all $k\in \Z^p \setminus \{0\}$ and $l'
\in \Z^{q'}$ with $|l'_1| + \cdots + |l'_{q'}| \leq 2$. We have
denoted $(\beta_1,...,\beta_{q'})$ by $\beta'$. {Let $\|\cdot\|_{s}$ be the $s$-analytic norm of an analytic function, \emph{i.e.}, the supremum norm of its analytic extension to the $s$-neighborhood of its (real) domain in the complexified space $\C^{p} \times \C^{p}/\Z^{p}$.}

\begin{theo} \label{KAM} Let $(\varpi^o,\beta^o) \in
  HD_{\bar\gamma,\bar\tau}$ and $N^o \in N_{\varpi^o,\beta^o}$. For
 some $d>0$ small enough,  there exists $\varepsilon>0$ such that for every
  Hamiltonian $N' \in C^\omega(\T^p \times B_\delta^{p+2q})$ such that
  $$\|N'-N^o\|_d \leq \varepsilon,$$
 there exists a vector $(\varpi,\beta)$ satisfying the following
  properties:
  \begin{itemize}
  \item the map $N' \mapsto (\varpi,\beta)$ is of class $C^\infty$ and
    is $\varepsilon$-close to $(\varpi^o,\beta^o)$ in the
    $C^\infty$-topology;
  \item if $(\varpi,\beta) \in HD_{\bar\gamma,\bar\tau}$, $N'$ is 
    symplectically analytically conjugate to a Hamiltonian $N \in
    N_{\varpi,\beta}$. 
  \end{itemize}
 Moreover, $\varepsilon$ can be chosen of the form $\hbox{Cst} \, \bar\gamma^{k}$
  (for some $\hbox{Cst}>0$, $k\geq 1$) when $\bar\gamma$ is small.
\end{theo}

This theorem is an analytic version of the $C^{\infty}$ ``hypothetical conjugacy
theorem'' of~\cite{FejozStability}. Its complete proof will appear in the article \cite{FejozMoser} of J. Féjoz. Actually, since analytic functions are $C^{\infty}$, except for the analyticity of the conjugation, other statements of the theorem directly follow from the ``hypothetical conjugacy theorem'' of~\cite{FejozStability}.

\subsection{An Iso-chronic KAM Theorem}

We now assume that the Hamiltonians $N^o=N^o_\iota$ and $N' = N'_\iota$
depend analytically (actually $C^1$-smoothly would suffice) on some parameter
$\iota \in B_1^{p+q}$. Recall that, for each $\iota$, $N^o_\iota$ is
of the form
\small
$$N^o_\iota= c^o_\iota + \langle \varpi^o_\iota,I \rangle + \sum_{j=1}^{q'}
\beta^o_{\iota,j} (x_j^2+ y_j^2) + \sum_{j=q'+1}^{q} \beta^o_{\iota,j}
(x_j^2 - y_j^2)+\langle A_{\iota,1}(\theta), I \otimes I \rangle +
\langle A_{\iota,2} (\theta), I \otimes Z \rangle + O_3(I,Z).$$\normalsize
Theorem~\ref{KAM} can be applied to $N^o_\iota$ and $N'_\iota$ for each
$\iota$.  We will now add some classical non-degeneracy condition to the hypotheses of the theorem, which ensure that the condition
``$(\varpi_\iota,\beta_\iota) \in HD_{\bar\gamma,\bar\tau}$'' actually
occurs often in the set of parameters.

Call 
$$HD^o = \left\{(\varpi^o_\iota,\beta^o_\iota) \in
  HD_{\bar\gamma,\bar\tau}: \, \iota \in 
  B_{1/2}^{p+q}\right\}$$ the set of ``accessible''
$(\bar\gamma,\bar\tau)$-Diophantine unperturbed frequencies. The
parameter is restricted to a smaller ball in order to avoid boundary
problems.

\begin{cor}[Iso-chronic KAM theorem] \label{cor: isochron nondeg} Assume the
  map
  $$B_1^{p+q} \rightarrow \R^{p+q}, \quad \iota \mapsto
  (\varpi_\iota^o,\beta_\iota^o)$$ is a diffeomorphism onto its image.
  If $\varepsilon$ is small enough and if $\|N'_\iota-N^o_\iota\|_{d} <
  \varepsilon$ for each $\iota$, the following holds:

  For every $(\varpi,\beta) \in HD^o$ there exists a unique $\iota \in
  B^{p+q}_1$ such that $N'_\iota$ is symplectically conjugate to some
  $N \in N_{\varpi,\beta}$. Moreover, there exists $\bar{\gamma}>0, \bar{\tau}>p-1$, such that the set
  $$\{\iota \in B_{1/2}^{p+q}: \, (\varpi_\iota,\beta_\iota) \in
  HD^o\}$$ has positive Lebesgue measure.
\end{cor}

\begin{proof}
  If $\varepsilon$ is small, the map $\iota \mapsto
  (\varpi_\iota,\beta_\iota)$ is $C^1$-close to the map $\iota \mapsto
  (\varpi^o_\iota,\beta^o_\iota)$ and is thus a diffeomorphism over
  $B_{2/3}^{p+q}$ onto its image, which contains the positive measure set $HD^o$ for some $\bar{\gamma}>0, \bar{\tau} \ge p-1$. The first
  assertion then follows from Theorem~\ref{KAM}. Since the inverse
  map $(\varpi,\beta) \mapsto \iota$ is smooth, it sends sets of positive measure onto sets of
  positive measure. 
\end{proof}

\begin{example-condition} \label{isochron nondeg}
When  $N^o=N^{o}(I)$ is integrable, $q=0$, we may set $N^o_\iota (I) := N^o(\iota + I)$. The iso-chronic non-degeneracy of $N^{o}_{\iota}$ is just the non-degeneracy of the Hessian $\mathscr{H}(N^{o})(I)$ of $N^{o}$  with respect to $I$: 
$$|\mathscr{H}(N^{o})(I)| \neq 0.$$
 When this is satisfied, Corollary \ref{isochron nondeg} asserts the persistence of a set of Lagrangian invariant tori of $N^o=N^{o}(I)$ parametrized by a positive measure set in the action space. By Fubini theorem, these invariant tori form a set of positive measure in the phase space. 

If the system $N^{o}(I)$ is properly degenerate, say 
$$I=(I^{(1)}, I^{(2)},\cdots, I^{(N)}),$$
and there exist real numbers
$$0<d_{1}<d_{2}<\cdots< d_{N}$$
such that
 $$N^{o}(I)=N^{o}_{1}(I^{(1)})+\epsilon^{d_{1}} N^{o}_{2}(I^{(1)}, I^{(2)}) + \cdots +\epsilon^{d_{N}} N^{o}_{N}(I),$$
 then, 
 $$|\mathscr{H}(N^{o})(I)| \neq 0, \,\, \forall\, 0<\epsilon<<1\,\Leftarrow|\mathscr{H}(N^{o}_{i})(I^{(i)})| \neq 0, \forall i=1,2,\cdots, N.$$
\emph{i.e.} the non-degeneracy of $N^{o}(I)$ can be verified separately at each scale.

Let us explain this fact by a simple example: Let $N^{o}(I_{1}, I_{2})=N^{o}_{1}(I_{1})+ \epsilon N^{o}_{2}(I_{1}, I_{2})$, then $$|\mathscr{H}(N^{o})(I_{1}, I_{2})|=\epsilon \cdot \frac{d^{2} N^{o}_{1}(I_{1})}{d I_{1}^{2}} \cdot \frac{d^{2} N^{o}_{2}(I_{1}, I_{2})}{d I_{2}^{2}} + O (\epsilon^{2}).$$ Therefore for small enough $\epsilon$, to have $|\mathscr{H}(N^{o})(I_{1}, I_{2})| \neq 0$, it suffices to have $\dfrac{d^{2} N^{o}_{1}(I_{1})}{d I_{1}^{2}} \neq 0$ and $\dfrac{d^{2} N^{o}_{2}(I_{1}, I_{2})}{d I_{2}^{2}} \neq 0$.

The smallest frequency of $N^{o}(I)$ is of order $\epsilon^{d_{N}}$. 
if $N^{o}(I)$ is non-degenerate, then for any $0<\epsilon<<1$, there exists a set of positive measure in the action space, such that under the frequency map, its image contains a set of positive measure of homogeneous Diophantine vectors in $HD_{\epsilon^{d_{N}} \bar{\gamma},\bar{\tau}}$ whose measure is uniformly bounded from below for $0<\epsilon<<1$. Actually, since for any vector $\nu' \in \R^{p+q}$,
$$\epsilon^{d_{N}} \nu' \in HD_{\epsilon^{d_{N}} \bar{\gamma},\bar{\tau}} \Leftrightarrow \nu' \in HD_{\bar{\gamma},\bar{\tau}}, $$
the measure of Diophantine frequencies of $N^{o}(I)$ in $HD_{\epsilon^{d_{N}} \bar{\gamma},\bar{\tau}}$ is lower bounded by the measure of Diophantine frequencies of 
$$N^{o}_{1}(I^{(1)})+N^{o}_{2}(I^{(1)}, I^{(2)}) + \cdots +N^{o}_{N}(I)$$
in $HD_{\bar{\gamma},\bar{\tau}}$, which is independent of $\epsilon$.

Following Theorem \ref{KAM}, we may thus set $\varepsilon=\hbox{Cst}\, (\epsilon^{d_{N}}\, \bar{\gamma})^{k}$ for the size of allowed perturbations, for some positive constant $\hbox{Cst}$ and some $k \ge 1$, provided $\bar{\gamma}$ is small.
\end{example-condition}

\subsection{Periodic Solutions Accumulating KAM Tori}

A theorem of J. P{\"o}schel (the last statement of \cite[Theorem 2.1]{Poeschel}; see also \cite{ChierchiaPeriodic}) permits us to show that there are families of periodic solutions accumulating the KAM Lagrangian tori. In our settings, this theorem can be stated in the following way:

\begin{theo} \label{Poeschel} Under the hypothesis of Corollary \ref{cor: isochron nondeg}, the Lagrangian KAM tori of the system $N'_{\iota}$ lie in the closure of the set of its periodic orbits.
\end{theo}

\section{Far from Collision Quasi-periodic Orbits of the Spatial Lunar Three-Body problem}

Now let us consider the Hamiltonian $F_{Kep}+\overline{F_{sec}^{n,n'}}+F_{secpert}^{n'+1}+F_{comp}^n$, seen as a system reduced by the SO(3)-symmetry. We now consider $\overline{F_{sec}^{n,n'}}$ as defined on a subset of the (SO(3)-reduced) phase space instead of the secular space, which has 4 degrees of freedom. 

To apply Corollary \ref{cor: isochron nondeg}, we start by verifying the non-degeneracy conditions in the system $F_{Kep}+\overline{F_{sec}^{n,n'}}$. As noted in Condition-Example \ref{isochron nondeg}, due to the proper degeneracy of the system, we just have to verify the non-degeneracy conditions in different scales.

Let us first consider the Kepler part:
$$F_{Kep}(L_{1}, L_{2})=-\dfrac{\mu_1^3 M_1^2}{2 {L_{1}}^2}-\dfrac{\mu_2^3 M_2^2}{2 L_{2}^2}.$$

Considered only as a function of $L_{1}$ and $L_{2}$, it is iso-chronically non-degenerate. 

To obtain the secular non-degeneracies of the system $ \overline{F_{sec}^{n, n'}}$, let us first consider the quadrupolar system $F_{quad}$. From Figure \ref{Fig: quadrupolar Phase LZ}, we see that for $C \neq G_2$, in the $(G_1,g_1)$-space, three types of regions are foliated by four kinds of closed curves of $F_{quad}(G_{1}, g_{1}; C, G_{2}, L_{1}, L_{2})$. They are regions around the elliptical singularities $B$
inside the separatrix of $A$ or $A'$, and the regions from $\{G_1=G_{1,max}\}$ and $\{G_1=G_{1,min}\}$ up to the nearest separatrix. These regions in turn correspond to three types of regions in the $(G_1,g_1, G_{2}, g_{2})$-space, foliated by invariant two-tori of the system $F_{quad}(G_{1}, g_{1}, G_{2}; C, L_{1}, L_{2})$. We build action-angle coordinates\footnote{See \cite{Arnold1989} for the method of building action-angle coordinates we use here.}, and let $\overline{\mathcal{I}}_1$ be an action variable in any one of these corresponding regions in the $(G_1,g_1)$-space. In Appendix \ref{non-degeneracySecularFrequencyMap}, we show that the quadrupolar frequency map is non-degenerate in a dense open set for almost all $\dfrac{C}{L_{1}}$ and $\dfrac{G_{2}}{L_{1}}$.  

Finally, for any fixed $C \neq 0$, the frequency map
$$(L_{1}, L_{2}, \overline{\mathcal{I}}_1, G_{2}) \mapsto \Bigl(\dfrac{\mu_1^3 M_1^2}{L_1^3},\,\dfrac{\mu_2^3 M_2^2}{L_2^3},\,\alpha^3 \nu_{quad, 1},\,\alpha^3 \nu_{quad, 2}\Bigr)$$
of $F_{Kep}+\alpha^{3} F_{quad}$ is a local diffeomorphism in a dense open set $\Omega$ of the phase space $\Pi$ symplectically reduced from the $\hbox{SO(3)}$-symmetry, in which $\nu_{quad, i}, i=1,2$ are the two frequencies of the quadrupolar system $F_{quad}(G_{1}, g_{1}, G_{2};C, L_{1}, L_{2})$ in the $(G_{i}, g_{i})$-plans respectively, which are independent of $\alpha$. 

For any $(n, n')$, the Lagrangian tori of the system $\overline{F_{sec}^{n,n'}}$ are $O(\alpha)$-deformations of Lagrangian tori of $\alpha^3 F_{quad}$. The frequency map of $F_{Kep}+\overline{F_{sec}^{n,n'}}$ are of the form
$$(L_{1}, L_{2}, \overline{\mathcal{J}}_1, G_{2}) \mapsto\Bigl(\dfrac{\mu_1^3 M_1^2}{L_1^3},\,\dfrac{\mu_2^3 M_2^2}{L_2^3},\,\alpha^3 \nu_{quad, 1}+O(\alpha^4),\,\alpha^3 \nu_{quad, 2}\,+O(\alpha^4)\Bigr),$$
which is thus non-degenerate in a open subset $\Omega'$ of $\Pi$ symplectically reduced from the $\hbox{SO(3)}$-symmetry for any choice of $n, n'$, with the relative measure of $\Omega'$ in $\Omega$ tends to 1 when $\alpha \to 0$, in which $\overline{\mathcal{J}}_1$ is defined analogously in the system $\overline{F_{sec}^{n,n'}}$ as  $\overline{\mathcal{I}}_1$ in $F_{quad}$. At the expense of restricting $\Omega'$ a little bit, we may further suppose that the transformation $\phi^{n} {\psi^{n' }}$ is well-defined. We fix $\alpha$ such that the set $\Omega'$ has sufficiently large measure in $\Omega$. 

In $\Omega'$, there exist $\bar{\gamma}>0, \bar{\tau} \ge 3$, such that the set of $(\alpha^{3} \bar{\gamma}, \bar{\tau})$-Diophantine invariant Lagrangian tori of $F_{Kep}+\overline{F_{sec}^{n, n'}}$ form a positive measure set whose measure is uniformly bounded for small $\alpha$ (Example-Condition \ref{isochron nondeg}).  By definition of $\Omega'$, near such a torus with action variables $(L_{1}^{0}, L_{2}^{0}, \overline{\mathcal{J}}_1^{0}, G_{2}^{0})$, there exists a $\lambda$-neighborhood for some $\lambda > 0$, such that the torsions of the Lagrangian tori of $F_{Kep}+\overline{F_{sec}^{n,n'}}$ do not vanish in this neighborhood. Let 
$$(L_{1}, L_{2}, \overline{\mathcal{J}}_1, G_{2})=\underline{\phi}^{\lambda}(L_{1}^{\lambda}, L_{2}^{\lambda}, \overline{\mathcal{J}}_1^{\lambda}, G_{2}^{\lambda}):=(L_{1}^{0} + \lambda \, L_{1}^{\lambda}, L_{2}^{0}+\lambda\, L_{2}^{\lambda}, \overline{\mathcal{J}}_1^{0}+\lambda \, \overline{\mathcal{J}}_1^{\lambda}, G_{2}^{0}+\lambda \, G_{2}^{\lambda}).
$$
Thus for any $(L_{1}^{\lambda}, L_{2}^{\lambda}, \overline{\mathcal{J}}_1^{\lambda}, G_{2}^{\lambda}) \in B_{1}^{4}$, and for any choice of $n, n'$,  the frequency map of the Lagrangian torus of $F_{Kep}+ \overline{F_{sec}^{n, n'}}$ corresponding to $(L_{1}^{0}+\lambda \, L_{1}^{\lambda}, L_{2}^{0}+\lambda \, L_{2}^{\lambda}, \overline{\mathcal{I}}_1^{0}+ \lambda \, \overline{\mathcal{I}}_1^{\lambda}, G_{2}^{0}+\lambda \, G_{2}^{\lambda})$ is non-degenerate. The existence of $\lambda$ follows from the definition of $\Omega'$.

We may now apply Corollary \ref{cor: isochron nondeg} for Lagrangian tori (\emph{i.e.} $p=4,q=0$) near the torus of $F_{Kep}+\overline{F_{sec}^{n, n'}}$ with action variables $(L_{1}^{0}, L_{2}^{0}, \overline{\mathcal{J}}_1^{0}, G_{2}^{0})$. We take $N'=\underline{\phi}^{\lambda*}{\psi^{n' *}} \phi^{n*} F $ (See Section \ref{SecularAndSecularIntegrableSystems} for definition of $\psi^{n}$ and $\phi^{n'}$), $N^{o}=\underline{\phi}^{\lambda*} (F_{Kep}(L_{1}, L_{2})+\overline{F_{sec}^{n,n'}} (L_{1}^{\lambda}, L_{2}^{\lambda}, \overline{\mathcal{J}}_1^{\lambda}, G_{2}^{\lambda}; C))$, with parameter $(L_{1}, L_{2}, \overline{\mathcal{J}}_1, G_{2}) \in B_{1}^{4}$ and perturbation $\underline{\phi}^{\lambda*} (F_{secpert}^{n'+1}+F_{comp}^n)$, whose order of smallness with respect to $\alpha$ can be made arbitrarily high by choosing large enough integers $n$ and $n'$. In particular, we may choose large enough $n$ and $n'$ so that Corollary \ref{cor: isochron nondeg} is applicable. 

The existence of an invariant Lagrangian torus of $\underline{\phi}^{\lambda*}\psi^{n' *} \phi^{n*} F $ (and thus of $F$) close to the Lagrangian torus of $F_{Kep}+\overline{F_{sec}^{n, n'}}$ with action variables $(L_{1}^{0}, L_{2}^{0}, \overline{\mathcal{J}}_1^{0}, G_{2}^{0})$ thus follows. We apply Corollary \ref{cor: isochron nondeg} near other  $(\alpha^{3} \bar{\gamma}, \bar{\tau})$-Diophantine invariant Lagrangian tori of $F_{Kep}+\overline{F_{sec}^{n, n'}}$ in $\Omega'$ analogously.

{In such a way, we get a set of positive measure of Lagrangian tori in the perturbed system $N'=\underline{\phi}^{\lambda*}{\psi^{n' *}} \phi^{n*} F$ (and thus of $F$) for any fixed $C > 0$. It remains to show that most of these Lagrangian tori stay away from the collisions. The transformations we have used to build the secular and secular-integrable systems are of order $O(\alpha)$, which shall bring an $O(\alpha)$-deformation to the collision set. Therefore, for fixed $C$ and $G_{2}$ (independent of $\alpha$), most of these invariant Lagrangian tori stay away from the collision set, provided $\alpha$ is small enough.}

\begin{theo}\label{FarFromCollisionMotions} For each fixed $C$, there exists a positive measure of 4-dimensional Lagrangian tori in the spatial three-body problem reduced by the SO(3)-symmetry, which are small perturbations of the corresponding Lagrangian tori of the system $F_{Kep}+ \alpha^3 F_{quad}$ reduced by the SO(3)-symmetry. 
\end{theo}

By rotation around $\vec{C}$, we obtain a positive measure of 5-dimensional invariant  tori in the lunar spatial three-body problem.

We establish the following types of quasi-periodic motions in the spatial lunar three-body problem (the required non-degeneracy conditions are collected in Appendix \ref{non-degeneracySecularFrequencyMap}):
\begin{itemize}
\item Motions along which $g_1$ librate around $\dfrac{\pi}{2}$, corresponds to the phase portraits around the elliptical singularity $B$; 
\item Motions along which $G_1$ remains large (eventually near $\{G_{1}=G_{1, max}\}$) while $g_1$ decreases;
\item Motions along which $G_1$ remains large (eventually near $\{G_{1}=G_{1, max}\}$) while $g_1$ increases;
\item Motions along which $G_1$ remains small but bounded from zero, while $g_1$ increases.
\end{itemize}
 
From Theorem \ref{Poeschel}, we get

\begin{theo}\label{AccumulatingPeriodicOrbits} There exist periodic orbits accumulating each of the KAM tori thus established in the spatial three-body problem reduced by the SO(3)-symmetry.
\end{theo}

Let us now consider the elliptic isotropic tori corresponding to the elliptic singularity $B$. Set $p=3, q=1$. The frequency of an elliptical isotropic torus with parameters $(L_1,L_2,G_2,C)$ corresponding to the only elliptic quadrupolar singularity $B$ in Figure \ref{Fig: quadrupolar Phase LZ} is of the form
$$\Bigl(\dfrac{\mu_1^3 M_1^2}{L_1^3},\,\dfrac{\mu_2^3 M_2^2}{L_2^3},\,\alpha^3 \nu_{quad, 2}+O(\alpha^4),\,\alpha^3 \nu_{quadn, G_2}\,+O(\alpha^4)\Bigr),$$
in which $\nu_{quadn, G_2}$ denotes the quadrupolar normal frequency of the elliptical isotropic torus. We show in Appendix \ref{non-degeneracySecularFrequencyMap} that the quadrupolar frequency map $$(G_2,C) \to (\nu_{quad, 2},\, \nu_{quadn, G_2})$$ is non-degenerate for almost all $\dfrac{C}{L_{1}}, \dfrac{G_{2}}{L_{1}}$. Set $C=C^{0}+\lambda C^{\lambda}$. We may now apply Corollary \ref{cor: isochron nondeg} in the same way as for Lagrangian tori, with parameters $L_{1}^{\lambda}, L_{2}^{\lambda}, G_{2}^{\lambda}, C^{\lambda}$ to obtain a positive 4-dimensional Lebesgue measure set of 3-dimensional isotropic elliptic tori in the direct product of the phase space of the reduced system
 of the spatial three-body problem (by the SO(3)-symmetry) with the space of parameters $C$. Let us call this 4-dimensional Lebesgue measure a ``product measure''.

\begin{theo}\label{elliptictori} There exists a positive product measure of 3-dimensional isotropic elliptic tori in the spatial three-body problem reduced by the SO(3)-symmetry, which are small perturbations of the isotropic tori corresponding to the elliptic secular singularity of $F_{Kep}+ \alpha^3 F_{quad}$ reduced by the SO(3)-symmetry. They give rise to 4-dimensional isotropic tori of the spatial three-body problem. 
\end{theo}

\appendix

\section{Estimates of the Perturbing Functions}  \label{section:estimates}

We first recall some hypothesis and notations from the beginning of Section \ref{SecularAndSecularIntegrableSystems}: 
\begin{itemize}
\item the masses $m_{0}, m_{1}, m_{2}$ are fixed arbitrarily;
\item Let $e_{1}^{\vee} <  e_{1}^{\wedge}, e_{2}^{\vee} < e_{2}^{\wedge}$ be positive numbers. We assume that
$$0<e_{1}^{\vee} < e_{1} < e_{1}^{\wedge}<1,\quad 0<e_{2}^{\vee} < e_{2} < e_{2}^{\wedge}<1.$$

\item  Let $a_{1}^{\vee} < a_{1}^{\wedge}$ be two positive real numbers. We assume that
$$a_{1}^{\vee} < a_{1} < a_{1}^{\wedge};$$
\item $\alpha=\dfrac{a_{1}}{a_{2}} < \alpha^{\wedge}:=\min\{\dfrac{1-e_{2}^{\wedge}}{80}, \dfrac{1-e_{2}^{\wedge}}{2 \sigma_{0}}, \dfrac{1-e_{2}^{\wedge}}{2 \sigma_{1}}\};$
\end{itemize}
From the relations $$a_{i} (1-e_{i}) \le \|Q_{i}\| \le a_{i} (1+e_{i}) \le 2 a_{i},$$
we obtain in particular that $\dfrac{\|Q_1\|}{\,\|Q_2\|} \le \dfrac{1}{\hat{\sigma}},$ for $\hat{\sigma}=\max\{\sigma_{0}, \sigma_{1}\}$.
\begin{lem}(Lemma 1.1 in \cite{QuasiMotionPlanar}) \label{Legendre Expansion} The expansion 
$$F_{pert}= - \mu_1 m_2 \sum^{\phantom{ssss}}_{n \ge 2} \sigma_n P_n(\cos{\zeta}) \dfrac{\|Q_1\|^n}{\, \, \, \|Q_2\|^{n+1}}$$
 is convergent in $\dfrac{\|Q_1\|}{\,\|Q_2\|} \le \dfrac{1}{\hat{\sigma}},$ (and therefore when $\alpha < \alpha^{
 \wedge}$) where $P_n$ is the n-th Legendre polynomial, $\zeta$ is the angle between the two vectors $Q_{1}$ and $Q_{2}$, $\hat{\sigma}= max\{\sigma_0,\sigma_1\}$ and $ \sigma_n= \sigma_0^{n-1}+(-1)^n \sigma_1^{n-1}$.
\end{lem}

As in \cite{QuasiMotionPlanar}, we have:

\begin{lem} \label{Perturbing estimate}
$$|F_{pert}| \le \hbox{Cst}\, \alpha^3,$$
for some constant Cst only depending on $m_{0}, m_{1}, m_{2}, e_{1}^{\vee}, e_{1}^{\wedge}, e_{2}^{\vee}, e_{2}^{\wedge}$.
\end{lem}        
\begin{proof} Since (\cite[p.129]{Kellogg}) 
$$|P_n(\cos{\zeta})| \le (\sqrt{2}+1)^{n} \le 3^{n},$$
and 
\begin{align*}
|\sigma_{n}| & = |\sigma_0^{n-1}+(-1)^n \sigma_1^{n-1}|\\
                   & \le |\sigma_{0}^{n-1}|+|\sigma_{1}^{n-1}|\\
                 & = \dfrac{m_{0}^{n-1}}{(m_{0}+m_{1})^{n-1}} +  \dfrac{m_{1}^{n-1}}{(m_{0}+m_{1})^{n-1}}< 1,
\end{align*} 
we have \small
\begin{align*} 
|F_{pert}|&=  \mu_1 m_2 \left|\sum^{\phantom{ssss}}_{n \ge 2} \sigma_n P_n(\cos{\zeta}) \dfrac{\|Q_1\|^n}{\, \, \, \|Q_2\|^{n+1}}\right|\\
               & \le \mu_{1} m_{2} \sum^{\phantom{ssss}}_{n \ge 2} 3^{n} \dfrac{\|Q_1\|^n}{\, \, \, \|Q_2\|^{n+1}}\\
               & \le \dfrac{\mu_{1} m_{2}}{a_{1}^{\vee} 3 (1-e^{\vee}_{1})} \sum^{\phantom{ssss}}_{n \ge 2}  \dfrac{3^{n+1} \alpha^{n+1}}{\, \, \, (1-e_{2}^{\wedge})^{n+1}}\\
               & \le \dfrac{\mu_{1} m_{2}}{a_{1}^{\vee} 3 (1-e^{\vee}_{1})}  \dfrac{3^{3} \alpha^{3}}{\, \, \, (1-e_{2}^{\wedge})^{2}} \dfrac{1}{1-e_{2}^{\wedge}- 3 \alpha}.
\end{align*}\normalsize
The conclusion thus follows when $\alpha< \dfrac{1-e_{2}^{\wedge}}{6}$. In particular, the constant Cst is uniform in the region of the phase space given by the hypothesis in the beginning of this appendix.
\end{proof}

In the following lemma, we regard $F_{pert}$ as a function of Delaunay variables 
$$(L_{1}, l_{1}, L_{2}, l_{2}, G_{1}, g_{1}, G_{2}, g_{2}, H_{1}, h_{1}, H_{2}, h_{2}) \in \mathcal{P}^* \subset \T^6 \times \R^{6},$$
in which $\mathcal{P}^{*}$ is defined, with the hypothesis of this appendix, by further asking that all the Delaunay variables are well defined. All variables are considered as complex, thus $\mathcal{P}^{*}$ is a subset of $T_{\C}=\C^6/\Z^6 \times \C^6$. The modulus of a complex number is denoted by $| \cdot |$ .

\begin{lem}\label{Appendix A: lem 3 ddd} There exists a positive number $s>0$, such that $| F_{pert} | \le \hbox{Cst}\, |\alpha|^{3} $ in the $s$-neighborhood $T_{\mathcal{P}^{*},s}$ of $\mathcal{P}^{*}$  for some constant Cst independent of $\alpha$.
\end{lem}
\begin{proof} By continuity, there exists a positive number $s$, such that in $T_{\mathcal{P}^{*},s}$, we have uniformly
$$
|\cos{\zeta}| \le 2;\,\,
\left|\dfrac{1}{\|Q_{1}\|}\right| \le \dfrac{2}{a_{1}^{\vee}(1-e^{\vee})};\,\,
\left|\dfrac{\|Q_{1}\|}{\|Q_{2}\|} \right| \le \dfrac{4 |\alpha|}{1-e_{2}^{\wedge}}.
$$
in which $\cos{\zeta}$, $\|Q_{1}\|$ and $\|Q_{2}\|$ are considered as the corresponding analytically extensions of the original functions.

Using Bonnet's recursion formula of Legendre polynomials
$$(n+1) P_{n+1}(\cos{\zeta})=(2n+1) \, \cos{\zeta}\, P_{n}(\cos{\zeta})-n P_{n-1}(\cos{\zeta}), $$
by induction on $n$, we obtain $|P_{n}(\cos{\zeta})| \le 5^{n}$.

Thus 
\begin{align*} 
|F_{pert}|&=  \mu_1 m_2 \left|\sum^{\phantom{ssss}}_{n \ge 2} \sigma_n P_n(\cos{\zeta}) \dfrac{\|Q_1\|^n}{\, \, \, \|Q_2\|^{n+1}}\right|\\
               & \le \mu_{1} m_{2}\left|\dfrac{1}{\|Q_{1}\|}\right| \sum^{\phantom{ssss}}_{n \ge 2} 5^{n} \left|\dfrac{\|Q_1\|}{ \, \|Q_2\|}\right|^{n+1}\\
               & \le \dfrac{\mu_{1} m_{2}}{a_{1}^{\vee}5 (1-e^{\vee}_{1})} \sum^{\phantom{ssss}}_{n \ge 2}  \dfrac{5^{n+1} 4^{n+1} |\alpha|^{n+1}}{\, \, \, (1-e_{2}^{\wedge})^{n+1}}\\
               & \le \dfrac{\mu_{1} m_{2}}{a_{1}^{\vee} 5 (1-e^{\vee}_{1})}  \dfrac{20^{3} |\alpha|^{3}}{\, \, \, (1-e_{2}^{\wedge})^{2}} \dfrac{1}{1-e_{2}^{\wedge}- 20 |\alpha|}.
\end{align*}\normalsize
It is then sufficient to impose $\alpha \le \alpha^{\wedge}$ and $s$ small enough to ensure that $|\alpha| \le \dfrac{1-e_{2}^{\wedge}}{40}$. 
\end{proof}

\section{Singularities in the Quadrupolar System} \label{MorseSingularities}

In this appendix, we show that,  for a dense open set of values of parameters $(G_{2}, C, L_{1}, L_{2})$, the singularities $A, B, A', E$ of $F_{quad}(G_{1}, g_1; G_{2}, C, L_{1}, L_{2})$ are of Morse type in the $(G_{1}, g_1)$-space.

Following \cite{LidovZiglin}, we define the normalized variables\footnote{in \cite{LidovZiglin}, it is $\bbdelta^{2}$ (was denoted by $\varepsilon$) which is taken as part of the coordinates.} 
$$\bbalpha=\dfrac{C}{L_1},  \bbbeta=\dfrac{G_2}{L_1}, \bbdelta=\dfrac{G_{1}}{L_{1}}, \bbomega=g_1.$$ 

From section \ref{QuadrupolarDynamics}, we deduce
$$F_{quad}=-\dfrac{k}{\bbbeta^{3}}(\mathcal {W}+\dfrac{5}{3}),$$
in which
$$\mathcal{W}(\bbdelta,\bbomega;\bbalpha,\bbbeta)=-2 \bbdelta^{2}+\dfrac{(\bbalpha^2-\bbbeta^2-\bbdelta^{2})^2}{4 \bbbeta^2} +5 (1-\bbdelta^{2}) \sin^2(\bbomega) (\dfrac{(\bbalpha^2-\bbbeta^2-\bbdelta^{2})^2}{4 \bbbeta^2 \bbdelta^{2}}-1).$$
The coefficient $k$ is independent of $\bbdelta$ and $\bbomega$ and $\bbalpha, \bbbeta$. We shall work with $\mathcal{W}$ from now on.

\begin{lem} \label{Quadrupolar Singularity non-degeneracy} For a dense open set of values of the parameters $(\bbalpha,\bbbeta)$ with $\bbalpha \neq \bbbeta$, all the singularities of the 1-degree of freedom Hamiltonian $\mathcal{W}$ (seen as a function of $(\bbdelta, \bbomega)$) are of Morse type. 
\end{lem}

\begin{proof} A singularity is of Morse type if, by definition, the Hessian of $\mathcal{W}$ at this point is non-degenerate. By evaluating the determinant of the Hessian of $\mathcal{W}$ with respect to $\bbdelta, \bbomega$ at the corresponding singularity, we get an analytic function of $\bbalpha,\bbbeta$, hence we only need to show that this function is not identically zero. Some of the following results were assisted by Maple 16.

Singularity $A$: The determinant of the Hessian of $W$ at this point is 
$$\dfrac{20 (\bbalpha^{2}+3\bbbeta^{2}-1)(\bbalpha^{2}- \bbbeta^{2})}{\bbbeta^{2} }<0.$$

Singularities $B$ and $A'$: The squares $\bbdelta_{B}^{2}$ of the ordinates $\bbdelta_{B}$ of $B$ and $A'$  are both determined by the same cubic equation
\small
\begin{equation}\label{Eq: cubic B}
x^{3}-(\dfrac{\bbbeta^{2}}{2}+\bbalpha^{2}+\dfrac{5}{8}) x^{2}+\dfrac{5}{8} (\bbalpha^{2}-\bbbeta^{2})^{2}=0.
\end{equation}
\normalsize
In order to make the analysis simple, we set the ordinate of B to $\dfrac{\sqrt{2}}{2}$ and the ordinate of $A'$ to $\dfrac{\sqrt{3}}{2}$. This leads to
\footnotesize
$$\bbalpha=\sqrt{\frac{13}{60}-\frac{\sqrt{2}}{10}}, \quad \bbbeta=\sqrt{\frac{13}{60}+\frac{\sqrt{2}}{10}},$$
\normalsize
which are in the allowable range of values (see Condition (3), Section \ref{QuadrupolarDynamics}). 

The determinants of the Hessian of $W$ at $B:(\bbdelta=\sqrt{2}/2,\bbomega=\pi/2)$ and $A':(\bbdelta=\sqrt{3}/2,\bbomega=\pi/2)$ are respectively $-\frac{51(30\sqrt{2}-5)}{8 (13+12 \sqrt{2})^{2}}$ and $-\frac{7(49560\sqrt{2}-61343)}{512 (13+12 \sqrt{2})^{2}}$. 

Singularity E($\bbalpha+\bbbeta \le 1$): at $\bbbeta=\bbalpha$, the determinant of the Hessian of $W$ at this point is $-\dfrac{10(2 \bbalpha - 1)^{2} (2 \bbalpha-3)}{\bbalpha^{2}}$. 
              
Singularity E($\bbalpha+\bbbeta>1$):  the determinant of the Hessian of $W$ at this point is 
$$-\dfrac{2(3 \bbbeta^{2}-\bbalpha^{2} - 1) (5 \bbalpha^{4}+5\bbbeta^{4}-10 \bbalpha^{2} \bbbeta^{2}-8 \bbalpha^{2}-4\bbbeta^{2}+3)}{\bbbeta^{4}}.$$                       
\end{proof}

In coordinates $(G_{1}, g_1)$, the circle $\{G_{1}=G_{1, min}\}$ corresponds to coplanar motions, and is therefore invariant under any system $\overline{F_{sec}^{n,n'}}$. There are no other singularities near $\{G_{1}=G_{1, min}\}$. Therefore, locally near $\{G_{1}=G_{1, min}\}$ in the 2-dimensional reduced secular space, the flow of $\overline{F_{sec}^{n,n'}}$ is orbitally conjugate to $F_{quad}$. 

\section{Non-degeneracy of the Quadrupolar Frequency Maps} \label{non-degeneracySecularFrequencyMap}

In this appendix, we verify the non-degeneracy of the frequency maps for the quadrupolar system $F_{quad} (G_{1}, g_1, G_{2}; C, L_{1}, L_{2})$ reduced by the $\hbox{SO(3)}$-symmetry (but keep the $\hbox{SO(2)}$-symmetry conjugate to $G_{2}$ unreduced). The calculations is assisted by Maple 16.

We continue to work in the normalized coordinates of  \cite{LidovZiglin}, described at the beginning of Appendix \ref{MorseSingularities}, \emph{i.e.} $$\bbalpha=\dfrac{C}{L_1}, \bbbeta=\dfrac{G_2}{L_1}, \bbdelta=\dfrac{G_{1}}{L_{1}}, \bbomega=g_1.$$ In these coordinates, we have $F_{quad}=\dfrac{k}{\beta^3}(\mathcal{W}+\dfrac{5}{3})$, and
\small
$$\mathcal{W}=-2 \bbdelta^2+\dfrac{(\bbalpha^2-\bbbeta^2-\bbdelta^2)^2}{4 \bbbeta^2} +5 (1-\bbdelta^2) \sin^2 \bbomega \, \left(\dfrac{(\bbalpha^2-\bbbeta^2-\bbdelta^2)^2}{4 \bbbeta^2 \bbdelta^2}-1\right). $$
\normalsize
Let $\overline{\mathcal{W}}(\bbdelta, \bbomega, \bbbeta;\bbalpha)=\dfrac{\mathcal{W}+\frac{5}{3}}{\beta^3}$. This function is now considered as a two degrees of freedom Hamiltonian defined on the 4-dimensional phase space, whose coordinates are $(\bbdelta, \bbomega, \bbbeta, g_{2})$, depending on the parameter $\bbalpha$. We shall formulate our results in terms of $\overline{\mathcal{W}}$, from which the corresponding results for $F_{quad}$ follow directly.

The main idea in the forthcoming proofs is to deduce the existence of torsion of $\overline{\mathcal{W}}$ from a local approximation system $\overline{\mathcal{W}}'(\bbdelta, \bbomega, \bbbeta; \bbalpha)$ whose flow, for fixed $\bbbeta$, is linear in the $(\bbdelta, \bbomega)$-plane. By analyticity, the torsion of $\overline{\mathcal{W}}$ is then non-zero almost everywhere in the corresponding region of the phase space foliated by the continuous family of the Lagrangian tori. 

To obtain the approximating system $\overline{\mathcal{W}}'$, we consider the reduced system $\widetilde{\mathcal{W}}$ of $\overline{\mathcal{W}}$ by fixing $\bbbeta$ and reduced by the $\hbox{SO(2)}$-action conjugate to $\bbbeta$. We either develop $\widetilde{\mathcal{W}}$ into Taylor series of $(\bbdelta, \bbomega)$ at an elliptic singularity and truncate at the second order, or develop $\widetilde{\mathcal{W}}$ into Taylor series of $\bbdelta$ at $\bbdelta=\hbox{Cst}$ and truncate at the first order. In both cases, the torsion of the truncated system amounts to the non-trivial dependence of a certain function of the coefficients of the truncation with respect to $\bbbeta$. 

\begin{lem} \label{Appendix C, lem 1} For a dense open set of values of $\bbalpha$, the frequency mapping of the Lagrangian tori of $\overline{\mathcal{W}}$ is non-degenerate on a dense open subset of the phase space of $\overline{\mathcal{W}}$. 
\end{lem}
\begin{proof}
By analyticity of the system, we just have to verify the non-degeneracy in small neighborhoods of the singularity $B$ and $\{\bbdelta=\bbdelta_{min}\}$ or $\{\bbdelta=\bbdelta_{max}=\max\{1, \bbalpha+\bbbeta\}\}$ for the system $\widetilde{\mathcal{W}}$.

In a small neighborhood of $B$ (whose $\bbdelta$-coordinate is denoted by $\bbdelta_{B}$), let $\bbdelta_1=\bbdelta-\bbdelta_B$, $\bbomega_1=\bbomega-\dfrac{\pi}{2}$. We develop $\widetilde{\mathcal{W}}$ into Taylor series of $\bbdelta_{1}$ and $\bbomega_{1}$:
  $$\widetilde{\mathcal{W}}=\Phi(\bbalpha,\bbbeta) + \Xi(\bbalpha,\bbbeta)\,\bbdelta_{1}^2 + \Upsilon(\bbalpha,\bbbeta)\,\bbomega_{1}^2+O\left((|\bbdelta_{1}|^2+|\bbomega_{1}|^2)^{\frac{3}{2}} \right). $$
In which 
\begin{align*}
&\Xi(\bbalpha,\bbbeta)=\dfrac{4 \bbbeta^{2} \bbdelta_{B}^{4}-24 \bbdelta_{B}^{6}+8 \bbdelta_{B}^{4} \bbalpha^{2}+5 \bbdelta_{B}^{4}+15 \bbalpha^{4}-30 \bbalpha^{2} \bbbeta^{2}+15 \bbbeta^{4}}{4 \bbbeta^{5} \bbdelta_{B}^{4}};\\
              &\Upsilon(\bbalpha,\bbbeta)=-\dfrac{5 (\bbdelta_{B}^{2}-1)((\bbalpha+\bbbeta)^{2}-\bbdelta_{B}^{2})((\bbalpha-\bbbeta)^{2}-\bbdelta_{B}^{2})}{4 \bbbeta^{5} \bbdelta_{B}^{4}}.
\end{align*}

From Equation \ref{Eq: cubic B}, we see that $\Upsilon(\bbalpha,\bbbeta) \not\equiv 0$. To show that $\Xi(\bbalpha,\bbbeta)\not\equiv 0$, we just need to use the identity (deduced from Equation \ref{Eq: cubic B}) 
$$15(\bbalpha^{2}-\bbbeta^{2})^{2}=24\bigl(\dfrac{\bbbeta^{2}}{2}+\bbalpha^{2}+\dfrac{5}{8}\bigr)\bbdelta_{B}^{4}-24 \bbdelta_{B}^{6}$$
to write $\Xi(\bbalpha,\bbbeta)$ into the form 
$$\Xi(\bbalpha,\bbbeta)=\dfrac{4  (\bbbeta^{2}+2 \bbalpha^{2}+\dfrac{5}{4}-\bbdelta_{B}^{2})}{\bbbeta^{5} }.$$

Since the singularity $B$ is elliptic, we have $\Xi(\bbalpha,\bbbeta) \Upsilon(\bbalpha,\bbbeta)>0$ for a dense open set of $(\bbalpha, \bbbeta)$. For $f$ close to $\Phi(\bbalpha,\bbbeta)$ when $\Xi >0$ (resp. $\Xi < 0$), the equation $f=\Phi(\bbalpha,\bbbeta) + \Xi(\bbalpha,\bbbeta)\bbdelta^2 + \Upsilon(\bbalpha,\bbbeta)\bbomega^2$ defines an ellipse in the $(\bbespilon,\bbomega)$-plane which bounds an area $\pi \dfrac{h-\Phi}{\sqrt{\Xi \Upsilon}}$, thus we may set $\overline{\mathcal{I}}_1= \dfrac{f-\Phi}{2 \sqrt{\Xi \Upsilon}}$, which is an action variable\footnote{See \cite{Arnold1989} for the method of building action-angle coordinates that we use here.} for the truncating system of $\widetilde{\mathcal{W}}$ up to second order of $\bbdelta_{1}$ and $\bbomega_{1}$. Therefore $W'=\Phi+ 2 \sqrt{\Xi \Upsilon}{}\, \overline{\mathcal{I}}_1+O(\overline{\mathcal{I}}_1^{\frac{3}{2}})$, where $O(\overline{\mathcal{I}}_1^{\frac{3}{2}})$ is a certain function of $\bbalpha, \bbbeta$ and $\overline{\mathcal{I}}_1$, which goes to zero not slower than $\overline{\mathcal{I}}_1^{\frac{3}{2}} \to 0$.

We denote by $\hbox{|Det|}\mathscr{H}(\mathfrak{F})$ the torsion of $\mathfrak{F}$ \emph{i.e.} the absolute value of the determinant of the Hessian matrix of a function $\mathfrak{F}(\overline{\mathcal{I}}_1,\bbbeta)$, \emph{i.e.}
  $$\hbox{|Det|}\mathscr{H}(\mathfrak{F}) \triangleq \left|\dfrac{\partial^2 \mathfrak{F}}{\partial \overline{\mathcal{I}}_1^2}  \dfrac{\partial^2 \mathfrak{F}}{\partial \bbbeta^2}- \Bigl(\dfrac{\partial^2 \mathfrak{F}}{\partial \overline{\mathcal{I}}_1 \partial \bbbeta}\Bigr)^2\right|.$$
 
It is direct to verify that 

$$\hbox{|Det|}\mathscr{H}(O(\overline{\mathcal{I}}_1^{\frac{3}{2}}))=O(\overline{\mathcal{I}}_1),$$
which is of at least the same order of smallness comparing to the quantity $f-\Phi$, which can be made arbitrarily small when restricted to small enough neighborhood of $B$, and 
$$\hbox{|Det|}\mathscr{H}(2 \sqrt{\Xi \Upsilon}{} \, \overline{\mathcal{I}}_1)=4 (\dfrac{\partial \sqrt{\Xi \Upsilon}}{\partial \bbbeta})^{2}.$$ 
This is exactly the torsion of the system $2 \sqrt{\Xi \Upsilon}{} \, \overline{\mathcal{I}}_1$ considered as a system of two degrees of freedom with coordinates $(\bbdelta, \bbomega, \bbbeta, \bar{g}_{2})$.

Therefore in order to prove the statement, it is enough to show that $\dfrac{\partial (\sqrt{\Xi \Upsilon}) }{\partial \bbbeta} \ne 0$\label{Not: quad action variables in the neighborhood of one of the elliptic equilibria} for some $\bbalpha$ and $\bbbeta$. 

Suppose on the contrary that the function $\sqrt{\Xi \Upsilon}$ is independent of $\bbbeta$, then the function $\Xi \Upsilon$ is also independent of $\bbbeta$. In view of the expressions of $\Xi$ and $\Upsilon$, this can happen only if one of the following expressions is a non-zero multiple of $\bbbeta^{\check{c}}$ for some integer $\check{c} \ge 1$:
$$\bbdelta_{B}^{2}-1, \quad \bbbeta^{2}+2 \bbalpha^{2}+\dfrac{5}{4}-\bbdelta_{B}^{2}, \quad \dfrac{1}{\bbdelta_{B}^{4}}, \quad(\bbalpha+\bbbeta)^{2}-\bbdelta_{B}^{2}, \quad (\bbalpha-\bbbeta)^{2}-\bbdelta_{B}^{2}$$
Since $\bbdelta_{B}^{2}$ solves Equation \ref{Eq: cubic B}, we substitute the particular form of $\bbdelta_{B}^{2}$ obtained in each case in Equation \ref{Eq: cubic B}, thus exclude the first two by comparing the constant term, exclude the third by comparing the lowest order term of $\bbbeta$, and exclude the last two by comparing the terms that only depends on $\bbalpha$.
As a result, $\dfrac{\partial (\sqrt{\Xi \Upsilon}) }{\partial \bbbeta}$ is non-zero for a dense open set of values of $(\alpha, \beta)$. 

We now consider the torsion of the tori near the lower boundary $\{\bbdelta=\bbdelta_{min}=|\bbalpha-\bbbeta|>0\}$.

Recall that the function $\widetilde{\mathcal{W}}$ and the coordinates $\bbdelta, \bbomega$ extend analytically to $\{0 < \bbdelta < \bbdelta_{min}\}$. This enables us to develop $\widetilde{\mathcal{W}}$ into Taylor series with respect to $\bbdelta$ at $\bbdelta=\bbdelta_{min}$: set $\bbdelta_{1}=\bbdelta-\bbdelta_{min}$, we obtain
$$\widetilde{\mathcal{W}}=\bar{\Phi}(\bbalpha,\bbbeta) + \bar{\Xi}(\bbalpha,\bbbeta, \bbomega)\,\bbdelta_{1}+O (\bbdelta_{1}^{2}),$$
in which 
$$\bar{\Xi}(\bbalpha,\bbbeta, \bbomega)=-\dfrac{2 \left((9 \bbalpha^2 \bbbeta-6 \bbalpha \bbbeta^2+\bbbeta^3-4 \bbalpha^3+5 \bbalpha)+(-5 \bbalpha+5 \bbalpha^3-10 \bbalpha^2 \bbbeta+5 \bbalpha \bbbeta^2)\cos^{2} \omega\right)}{\bbbeta^4 |\bbalpha-\bbbeta|}.$$

We eliminate the dependence of $\bbomega$ in the linearized Hamiltonian $\bar{\Phi}(\bbalpha,\bbbeta)+\bar{\Xi}(\bbalpha,\bbbeta, \bbomega)\,\bbdelta_{1}$ by computing action-angle coordinates. The value of the action variable $\overline{\mathcal{I}}_{1}$ on the level curve 
$$E_{f}: \bar{\Phi}(\bbalpha,\bbbeta)+\bar{\Xi}(\bbalpha,\bbbeta, \bbomega)\,\bbdelta_{1}=f$$
is computed from the area between this curve and $\bbdelta_{1}=0$, that is
$$\overline{\mathcal{I}}_{1}=\dfrac{1}{2 \pi} \int_{E_{f}} \bbdelta_{1} d \bbomega=\dfrac{f-\bar{\Phi}(\bbalpha,\bbbeta)}{2 \pi} \int_{0}^{2 \pi} \dfrac{1}{\bar{\Xi}(\bbalpha,\bbbeta, \bbomega)} d \bbomega=\overline{\mathcal{I}}_{1}.$$

 We have then
$$\widetilde{\mathcal{W}}=\bar{\Phi}(\bbalpha,\bbbeta) + 2 \pi\left(\int_{0}^{2 \pi} \frac{1}{\bar{\Xi}(\bbalpha,\bbbeta, \bbomega)} d \bbomega\right)^{-1}\,\overline{\mathcal{I}}_{1}+O (\overline{\mathcal{I}}_{1}^{2}).$$

As in the proof of Lemma \ref{Appendix C, lem 1}, for $\overline{\mathcal{I}}_{1}$ small enough, the torsion of $\widetilde{\mathcal{W}}$ is dominated by the torsion of the term linear in $\overline{\mathcal{I}}_{1}$, which is 

$$\left[2 \pi \frac{d}{d \bbbeta} \left(\int_{0}^{2 \pi} \frac{1}{\bar{\Xi}(\bbalpha,\bbbeta, \bbomega)} d \bbomega\right)^{-1}\right]^{2}$$

Using the formula
$$\int_{0}^{2 \pi} \dfrac{d \bbomega}{a + b \cos \bbomega}=\dfrac{2 \pi}{\sqrt{a^{2}-b^{2}}}$$
we obtain
$$2 \pi\left(\int_{0}^{2 \pi} \frac{1}{\bar{\Xi}(\bbalpha,\bbbeta, \bbomega)} d \bbomega\right)^{-1}=-\dfrac{2\sqrt{\bbalpha+\bbbeta} \sqrt{9 \bbalpha^2 \bbbeta-6\bbalpha \bbbeta^2+\bbbeta^3-4 \bbalpha^3+5 \bbalpha}}{\bbbeta^4},$$ 
which depends non-trivially on $\bbbeta$. Therefore the torsion of the system 
$$2 \pi\left(\int_{0}^{2 \pi} \frac{1}{\bar{\Xi}(\bbalpha,\bbbeta, \bbomega)} d \bbomega\right)^{-1}\,\overline{\mathcal{I}}_{1}$$
which is considered as a function of $\bbbeta, \overline{\mathcal{I}}_{1}$, is not identically zero. 

In the case $\bbdelta_{max}(=\min\{1,\bbalpha+\bbbeta\})=\bbalpha+\bbbeta$, since $\overline{\mathcal{W}}$ is an odd function of $\bbbeta$, we may simply replace $\bbbeta$ by $-\bbbeta$ in the formula for tori near $\bbdelta=\bbdelta_{min}$ presented above. The required non-degeneracy follows directly.

In the case $\bbdelta_{max}(=\min\{1,\bbalpha+\bbbeta\})=1$, by the same method, we only have to notice that the function
\footnotesize
\begin{align*}
\footnotesize
&\left(\bbbeta^5 \int_{0}^{2 \pi}\dfrac{d \omega}{(5 \bbalpha^4-10 \bbalpha^2\bbbeta^2-10\bbbeta^2+5-10\bbalpha^2+5\bbbeta^4)\cos^2{\bbomega}+(4\bbbeta^2+8\bbalpha^2-3-5\bbalpha^4+10\bbalpha^2\bbbeta^2-5\bbbeta^4)}\right)^{-1}\\  \normalsize&=\dfrac{\sqrt{(6 \bbbeta^{2}-3 \bbalpha^{2}-2+5 \bbalpha^{4})(\bbbeta^{2}+8 \bbalpha^{2}-3-5 \bbalpha^{4}+10 \bbalpha^{2} \bbbeta^{2})}}{\bbbeta^{5}}
\end{align*}
\normalsize
depends non-trivially on $\bbbeta$. 
\end{proof}

\begin{lem} \label{lowerdimensional non-degeneracy} The frequency map of the elliptic isotropic tori corresponding to the secular singularity $B$ is non-degenerate for a dense open set of values of $(\bbalpha, \bbbeta)$.
\end{lem}
\begin{proof} Following from the previous proof, we only need to note in addition that the secular frequency map of the elliptic isotopic tori corresponding to the secular singularity $B$ is the limit of the secular frequency map of the Lagrangian tori around $B$: At the limit, the frequency of these tori with respect to $\overline{\mathcal{I}}_{1}$ becomes the normal frequency of the lower dimensional secular tori corresponds to $B$, and the frequency with respect to $G_{2}$ becomes the tangential frequency of the lower dimensional tori. We see that the frequency of the approximating Hamiltonian $2 \sqrt{\Xi \Upsilon}{}\, \overline{\mathcal{I}}_1$ is independent of $\overline{\mathcal{I}}_1$, hence its frequency map for Lagrangian tori near the lower dimensional tori gives in the same time the frequency map for the lower dimensional tori. By the same reasoning and calculations as in the proof of Lemma \ref{Appendix C, lem 1}, the non-degeneracy condition of the secular frequency holds for a dense open set in the $(\bbalpha, \bbbeta)$-space. 
 \end{proof}

\begin{ack} The results of this article is part of my Ph.D. thesis \cite{ZLthesis} prepared at the Paris Observatory and the Paris Diderot University. Many thanks to my supervisors, Alain Chenciner and Jacques Féjoz, for their generous and crucial help during these years.
Sincerely thanks to Jesús Palacian for helpful comments on some preliminary version of this article.
\end{ack}

\bibliography{QuasiperiodicMotionSpatial}
\end{document}